\documentclass[12pt]{extarticle}
\pdfoutput=1
\usepackage[utf8]{inputenc}
\usepackage{caption}
\usepackage{xfrac} 
\oddsidemargin \evensidemargin
\usepackage[normalem]{ulem}

\usepackage{floatrow}
\usepackage{xcolor}
\usepackage{amsmath, amssymb, amsfonts}
\usepackage{mathtools}
\usepackage{enumerate}
\usepackage[all]{xy}
\usepackage[utf8]{inputenc}
\usepackage{bbm}
\usepackage{mathrsfs}
\usepackage{tikz,pgfplots,tikz-3dplot}
\usepackage{tikzscale}
\usepackage{subcaption}
\usepackage{autobreak}
\usepackage{kbordermatrix}

\renewenvironment{thebibliography}[1]{
  \begin{oldthebibliography}{#1}
    \setlength{\itemsep}{0.18em}
    \setlength{\parskip}{0em}
}
{
  \end{oldthebibliography}
}

\usepackage[english]{babel}
\usepackage{alphabeta}

\usepackage[colorlinks=true, allcolors=teal,linkcolor=teal, citecolor=teal,backref=page]{hyperref}

\newcommand{\var}{\alpha}

\definecolor{mycolor1}{rgb}{0.00000,0.44700,0.74100}
\definecolor{mycolor2}{rgb}{0.8500, 0.3250, 0.0980}
\definecolor{mycolor3}{rgb}{0.9290, 0.6940, 0.1250}
\definecolor{mycolor4}{rgb}{0.4940, 0.1840, 0.5560}
\definecolor{mycolor5}{rgb}{0.4660, 0.6740, 0.1880}

\newcommand{\C}{\mathbb{C}}
\newcommand{\Z}{\mathbb{Z}}

\usepackage{amsthm}
\usepackage{cleveref}

\newtheorem{theorem}{Theorem}[section]

\newtheorem{lemma}[theorem]{Lemma}
\newtheorem{corollary}[theorem]{Corollary}
\newtheorem{proposition}[theorem]{Proposition}

\theoremstyle{definition}

\newtheorem{examplex}[theorem]{Example}
\newenvironment{example}
{\pushQED{\qed}\examplex}
{\popQED\endexamplex}

\newtheorem{remarkx}[theorem]{Remark}
\newenvironment{remark}
{\pushQED{\qed}\remarkx}
{\popQED\endremarkx}

\numberwithin{equation}{section}

\newtheoremstyle{citing}% name
{}%     Space above, empty = `usual value'
{}%     Space below
{\itshape}% Body font
{}%     Indent amount (empty = no indent, \parindent = para indent)
{\bfseries}% Thm head font
{\textbf{.}}% Punctuation after thm head
{.5em}% Space after thm head: " " = normal interword space;
{\thmnote{#3}}% Thm head spec
{\theoremstyle{citing}
}

\DeclareMathOperator{\CC}{\mathbb{C}}
\DeclareMathOperator{\RR}{\mathbb{R}}
\DeclareMathOperator{\ZZ}{\mathbb{Z}}

\DeclareMathOperator{\Conv}{Conv}

\DeclareMathOperator{\vol}{vol}

\renewcommand{\d}{\mathrm{d}}
\newcommand{\n}{\mathsf{n}}

\newcommand{\eps}{\varepsilon}

\newcommand{\uu}{u}

\definecolor{mycolor1}{rgb}{0.0,0.53,0.74}
\definecolor{mycolor2}{rgb}{0.9,0.17,0.31}
\definecolor{mycolor3}{rgb}{1.0, 0.8, 0.0}
\definecolor{mycolor4}{rgb}{0.73, 0.2, 0.52}
\definecolor{mycolor5}{rgb}{0.0, 0.66, 0.47}
\definecolor{mycolor6}{rgb}{1.0, 0.6, 0.4}

\def\treeBis{\tikz[baseline=.1ex]{
\fill (0,0.2) circle (2pt) coordinate (A);
\draw (0,0) node [anchor=north][inner sep=0.75pt]  [font=\scriptsize]  {$X_{2} \!+\!Y_{12}$};
}
}

\def\treeOne{\tikz[baseline=.1ex]{
\fill (0,0.2) circle (2pt) coordinate (A);
\draw (0,0) node [anchor=north][inner sep=0.75pt]  [font=\scriptsize]  {$X_{1} \!+\!Y_{12}$};
}
}

\def\treeTwo{\tikz[baseline=.1ex]{
\fill (0,0.2) circle (2pt) coordinate (A);
\fill (1,0.2) circle (2pt) coordinate (B);
\draw (0,0) node [anchor=north][inner sep=0.75pt]  [font=\scriptsize]  {$X_{2} \!+\!Y_{12}$};
\draw (1,0) node [anchor=north][inner sep=0.75pt]  [font=\scriptsize]  {$X_{3}$};
\draw (0.5,0.3) node [anchor=south][inner sep=0.75pt]  [font=\scriptsize]  {$Y_{23}$};
\draw (A)--(B);
}
}

\def\treeThree{\tikz[baseline=.1ex]{
\fill (0,0.2) circle (2pt) coordinate (A);
\fill (1,0.2) circle (2pt) coordinate (B);
\draw (0,0) node [anchor=north][inner sep=0.75pt]  [font=\scriptsize]  {$X_{1}$};
\draw (1,0) node [anchor=north][inner sep=0.75pt]  [font=\scriptsize]  {$X_{2}\!+\!Y_{23}$};
\draw (0.5,0.3) node [anchor=south][inner sep=0.75pt]  [font=\scriptsize]  {$Y_{12}$};
\draw (A)--(B);
}
}

\def\treeFour{\tikz[baseline=.1ex]{
\fill (0,0.2) circle (2pt) coordinate (A);
\draw (0,0) node [anchor=north][inner sep=0.75pt]  [font=\scriptsize]  {$X_{3} \!+\!Y_{23}$};
}
}

\def\treeFive{\tikz[baseline=.1ex]{
\fill (0,0.2) circle (2pt) coordinate (A);
\draw (0,0) node [anchor=north][inner sep=0.75pt]  [font=\scriptsize]  {$X_{2} \!+\!Y_{12}\!+\!Y_{23}$};
}
}

\def\treeSix{\tikz[baseline=.1ex]{
\fill (0,0.2) circle (2pt) coordinate (A);
\draw (0,0) node [anchor=north][inner sep=0.75pt]  [font=\scriptsize]  {$X_{3} \!+\!Y_{23}$};
}
}

\def\treeSeven{\tikz[baseline=.1ex]{
\fill (0,0.2) circle (2pt) coordinate (A);
\draw (0,0) node [anchor=north][inner sep=0.75pt]  [font=\scriptsize]  {$X_{1} \!+\!Y_{12}$};
}
}

\def\treeEight{\tikz[baseline=.1ex]{
\fill (0,0.2) circle (2pt) coordinate (A);
\draw (0,0) node [anchor=north][inner sep=0.75pt]  [font=\scriptsize]  {$X_{2} \!+\!Y_{12}\!+\!Y_{23}$};
}
}

\pgfplotsset{compat=1.18}

\title{Euler Discriminant of Complements of Hyperplanes}
\author{Claudia Fevola and Saiei-Jaeyeong Matsubara-Heo}
\date{}

\begin{document}
\maketitle 

\begin{abstract}
The Euler discriminant of a family of very affine varieties is defined as the locus where the Euler characteristic drops. In this work, we study the Euler discriminant of families of complements of hyperplanes. We prove that the Euler discriminant is a hypersurface in the space of coefficients, and provide its defining equation in two cases: (1) when the coefficients are generic, and (2) when they are constrained to a proper subspace. In the generic case, we show that the multiplicities of the components can be recovered combinatorially. This analysis also recovers the singularities of an Euler integral. In the appendix, we discuss a relation to cosmological correlators.
\end{abstract}

\section{Introduction}
The study of Euler integrals has a long tradition dating back to works by Aomoto and Gelfand, among others. For an extensive overview, we refer the readers to  \cite{aomoto2011theory, aomoto1997twisted} and references therein. In this work, we are mainly interested in families of Euler integrals of the form %\saiei{Why linear forms?}
\begin{equation}\label{eq:integral_intro}
I_{\Gamma}(z)\, =\,\, \int_\Gamma \,  h_{k+1}(\alpha;z)^{s_1}\,\cdots \,h_n(\alpha;z)^{s_n} \, \alpha_1^{\nu_1}\,\cdots\, \alpha_k^{\nu_k} \,\frac{\mathrm{d} \alpha_1}{\alpha_1} \wedge \cdots \wedge \frac{\mathrm{d} \alpha_k}{\alpha_k},
\end{equation}
where $\alpha=(\alpha_1,\dots,\alpha_k)$ are coordinates on $(\mathbb{C}^*)^k$ and $h_{j}(\alpha;z) = \,z_{0j} \,+\, z_{1j}\alpha_1\,+\,\dots\,+ \,z_{kj}\alpha_k$ for $j=k+1,\dots,n$ are linear forms in the variables $\alpha$ and coefficients $z_{ij}\in \mathbb{C}$. The exponents $s_j,\nu_i$ take on complex values, making the integrand multivalued. Twisted de Rham (co)-homology provides a rigorous framework for defining Euler integrals. More precisely, the integration contour $\Gamma$ is a twisted $k$-cycle on the complement of hyperplanes
\begin{equation}\label{eq:complement_intro}
    X_z\, = \, (\CC^*)^k\setminus\left( \{h_{k+1}(\alpha;z)=0\}\cup \cdots \cup \{h_n(\alpha;z)=0\}\right),
\end{equation}
for fixed coefficients $z_{ij}\in \CC$ in each linear form. The term \textit{twisted} essentially means that $\Gamma$ also records the choice of which branch of the integrand to integrate. We refer to \cite[Section~3]{matsubara2023four} for an introduction to twisted cycles and co-cycles. When the coefficients are generic, the signed Euler characteristic $(-1)^k\cdot \chi(X_z)$ of the hyperplanes complement counts the local solutions of an \textit{$A$-hypergeometric system} and determines the dimension of the top-dimensional twisted de Rham cohomology group. For a discussion of these equalities in a broader framework, we refer to \cite[Theorem~1.1]{agostini2022vector}.  Furthermore, for $X_z$ as in \eqref{eq:complement_intro}, the signed Euler characteristic also coincides with the \textit{beta invariant} of $X_z$, and with the number of bounded regions if the arrangement is real. The latter is known as Varchenko’s conjecture \cite{varchenko1995critical}, proved by Orlik–Terao \cite{orlik1995number}. The value of $(-1)^k\cdot \chi(X_z)$ for arrangements in general position was computed in \cite{aomoto1997twisted} via an explicit computation of a twisted cohomology basis of logarithmic forms.

In \cite{fevola2024landau, fevola2024principal}, it is proposed to explore the variation of the Euler characteristic to determine the singularities of Feynman integrals, expressed in the form of Euler integrals with coefficients varying in a subspace of the space of physical parameters. %This idea was explored for the study of singularities of Feynman integrals in particle physics. 
This approach led to the re-discovery of the \textit{Euler discriminant}, which was originally introduced by Esterov \cite{esterov2013discriminant} in the context of studying singular solutions to polynomial systems. Here, we present a version of this definition adapted to our framework, with the general definition provided in Section~\ref{sec:2_EulerDiscriminant}. Let $z$ denote the matrix of size $(k+1)\times(n-k)$ whose columns are the coefficients of the hyperplanes $h_j(\alpha;z)$. We denote $Z\subset \CC^{(k+1)(n-k)}$ the subspace where these coefficients vary. The Euler discriminant $\nabla_{\chi}(Z)$ is the locus of points $z\in Z$ for which the Euler characteristic $(-1)^k\cdot \chi(X_z)$ is smaller than its generic value. 
When $Z$ is the entire coefficients space, the Euler discriminant coincides with the zero locus of the \textit{principal $A$-determinant}. This is a polynomial in the variables $z_{ij}$ from the theory of $A$-hypergeometric systems \cite{GKZbook}. We refer to Section~\ref{sec:PAD} for its definition.
In this article, we investigate the Euler discriminant associated to the complement of hyperplanes $X_z$ for two distinct choices of coefficients spaces:
\begin{enumerate}
    \item[$(1).$] $Z$ is $ \mathbb{C}^{(k+1)\times(n-k)}\cap \{z_{ij}=0 \,\,\text{for some } \, i=0,\dots,k,\, j=k+1,\dots,n\} $: in this case, the Euler discriminant coincides with the vanishing locus of a principal $A$-determinant. Theorem~\ref{thm:PAD} describes the defining polynomial, expressed as a product of $A$-discriminants corresponding to various faces of a polytope. Importantly, we address the combinatorial problem of detecting the faces of the polytope that do not contribute to a discriminant. We characterize them using a condition on an associated bipartite graph.
    \item[$(2).$] $Z$ is a general subvariety of $ \mathbb{C}^{(k+1)\times(n-k)}$: in applications, coefficients are often constrained to a subspace of the full coefficients space. This setup is more general than the previous case in the sense that the principal $A$-determinant is not defined for general subvarieties. Using the Orlik-Solomon algebra associated to the family $X_z$,  Theorem~\ref{thm:Euler_discr} describes the defining polynomial of the Euler discriminant once again through a determinantal formula.
\end{enumerate}

\noindent
As for Theorem \ref{thm:PAD}, we note that \cite[Proposition 4.4]{galashin2024trianguloids} determines the normal fan of the principal $A$-determinant, whereas Theorem \ref{thm:PAD} further provides an explicit description of its Newton polytope.
In both cases (1) and (2) the matroid associated to the hyperplane arrangement is connected to the Euler discriminant. A clear consequence of Theorems~\ref{thm:PAD} and~\ref{thm:Euler_discr} is that the Euler discriminant corresponds to locus where the matroid changes. This result aligns with \cite[Theorem~1.3]{clarke2023matroid}, where, motivated by the study of maximum likelihood (ML) degree in algebraic statistics, it was shown that the signed Euler characteristic is a matroid invariant. The paper also studies the resulting matroid stratification in several small-dimensional cases. An interesting problem would be to relate this stratification to the Euler stratification, defined in \cite{wiesmann2024euler}.

A key question raised in \cite{fevola2024principal} concerns the relationship between the Euler discriminant and the singular locus of the $D$-module annihilating the Feynman integral. This question is particularly relevant in scattering amplitudes, where studying the differential equations that annihilate Feynman integrals has been a deeply explored method to be able to evaluate these integrals. In Section~\ref{sec:SingularLocus}, we formalize this question in the case of Euler integrals by defining an appropriate $D$-module. Theorem~\ref{thm:Euler discriminant and singular locus} showes that the singular locus of such a $D$-module and the Euler discriminant coincide in their codimension one part. Furthermore, these loci are purely one-codimensional in the case of complements of hyperplanes, yielding the expected equality in this case, as stated in Theorem~\ref{thm:5.1}.

Notably, examples of Euler integrals have appeared in theoretical cosmology \cite{arkani2023, baumann2024kinematic}, where their integrands often takes the form of products of linear forms as in \ref{eq:integral_intro}. The Appendix to this work offers a self-contained introduction to cosmological integrals as mathematical objects. We will use these integrals as concrete examples of Euler integrals of linear forms with coefficients parametrized by physical parameters and compute their Euler discriminant. 
\textbf{Outline.} The structure of the paper is as follows. In Section~\ref{sec:2_EulerDiscriminant}, we introduce the general notions of Euler discriminants and the principal $A$-determinant for Euler integrals. We explore their relationship with the Cayley configuration (Theorem~\ref{thm:Cayley}). Section~\ref{sec:PAD_sparse} focuses on the principal $A$-determinant for the choice of $Z$ as in $(1)$. We recall the definition of edge polytope and associated bipartite graph, which are central to proving Theorem~\ref{thm:PAD}.  An implementation to compute the principal $A$-determinant from Theorem~\ref{thm:PAD} is available at
\begin{equation}\label{mathrepo}
\href{https://mathrepo.mis.mpg.de/EulerDiscriminantHyperplanes}{\texttt{https://mathrepo.mis.mpg.de/EulerDiscriminantHyperplanes}}
\end{equation} 
Subsection~\ref{sec:subdiagram_volume} presents a method for computing the multiplicities of the components. Section~\ref{sec:Euler_discriminant} examines the Euler discriminant for coefficients as in $(2)$. In Section~\ref{sec:SingularLocus}, we relate the Euler discriminant to the singular locus of a $D$-module, proving in Theorem 5.1 that these loci coincide in codimension one. Finally, the Appendix provides an introduction to cosmological integrals, illustrating their connection to Euler integrals of linear forms. This includes computational examples of Euler discriminants for cosmological integrals.

\section{Euler discriminant and principal A-determinant}

\subsection{Euler discriminant}\label{sec:2_EulerDiscriminant}
Let $f_0,\dots,f_\ell$ be Laurent polynomials, where each $f_i$ is given as $$f_i(\var;z)\,=\,\sum_{\uu\in A_i}z_{\uu,i}\var^\uu,$$ with variables $\var\in (\C^*)^n$ and fixed monomial support represented by a set $A_i\subset \ZZ^{n}$. The coefficients $z_i=(z_{\uu,i})_{\uu\in A_i}$ take values in $\CC^{A_i} \coloneqq \CC^{|A_i|}$, and  
we assume that $\bigcup_{i=0}^\ell A_i$ spans the ambient lattice $\Z^n$. Once coefficients $z_i\in \C^{A_i}$ are fixed, the Laurent polynomials $f_i(\var;z)$ define a hypersurface in the algebraic torus $(\C^*)^n$:
\begin{equation}
    V_{A,z}\, = \, V_{(\CC^*)^n}\biggl(\prod_{i=0}^\ell f_i(\,\,\cdot\,\,;z)\biggr) \,=\, \left\{\var \in (\CC^*)^n \,\, |\, \, \prod_{i=0}^\ell f_i(\var;z) = 0 \right\}
\end{equation}
\noindent
Here, $f(\,\cdot\,;z)$ stands for the Laurent polynomial $f(\var;z)$ viewed as a regular function on $\var\in(\CC^*)^n$.
Consider $ Z\subset \CC^A$, a smooth subvariety of the full parameter space $\CC^A \coloneqq \CC^{A_0}\times\cdots\times\CC^{A_\ell}$.
For each $z\in  Z$, we set
\begin{equation}\label{eq:Xz}
    X_z\,:=\,(\CC^*)^n\setminus V_{A,z}.
\end{equation}
\noindent
We denote $\chi_z$ the signed Euler characteristic $(-1)^n\cdot\chi\left( X_z\right)$, and $\chi^*$  the maximal value ${\rm max}\{\chi_z\mid z\in  Z\}$.
The \textit{Euler discriminant} $\nabla_{\chi}( Z)$ is  the locus
\begin{equation}\label{eq:2.3}
    \nabla_{\chi}( Z)\,:=\,\{ z\in Z\mid \chi_z<\chi^*\}.
\end{equation}
If follows from \cite[Theorem 3.1]{fevola2024principal}  that $\nabla_{\chi}(Z)$ is a closed subvariety of $Z$. 

\begin{remark}
Esterov first defined the Euler discriminant in \cite{esterov2013discriminant}, and it was later applied to Feynman integrals in \cite{fevola2024principal}. Esterov's definition, in particular, also takes into account the multiplicities of each component: the Euler discriminant in \cite[Definition~3.1]{esterov2013discriminant} is the divisor given by the formal sum of the closure of the codimension-one strata $P_i$ in $Z$ where the Euler characteristic $\chi_z$ for $z\in P_i$ is equal to a fixed value $\chi_i$. The multiplicity of each component is defined as the drop in Euler characteristic $\chi^*-\chi_i$. This definition is motivated by the notion of multiplicity, which was established for the generic setting where the subvariety $Z$ is the entire parameter space, c.f. Theorem~\ref{thm:drop_Euler}.  
\end{remark}

The \textit{Cayley configuration} of $\{A_i\}_{0\leq i \leq \ell}$ is represented by the columns of the $(n+\ell)\times \sum_{i=0}^\ell |A_i|$ matrix given by
\begin{equation}\label{eq:cayley}
\left( \begin{array}{ccc|ccc|c|ccc}
0 & \cdots & 0 & 1 & \cdots & 1 &  & 0 & \cdots & 0 \\
& \vdots & & & \vdots & & \cdots& & \vdots & \\
0 & \cdots & 0 & 0 & \cdots & 0 & & 1 & \cdots & 1\\
&A_0& & &A_1& &  & &A_\ell
\end{array} \right).
\end{equation}
The Cayley configuration also has a natural interpretation as the monomial support of a multivariate Laurent polynomial in the variables $\beta = (\beta_1,\dots,\beta_\ell)$ and  $\var$ as above, given by
\begin{equation}\label{eq:Cayley_poly}
    f (\var,\beta) \,\coloneqq\, f_0\, + \,\sum_{i=1}^\ell \beta_i f_i(\alpha). 
\end{equation}
The following result relates the Euler characteristic of the variety defined by the Laurent polynomials $f_0,\dots,f_\ell$ and the one given by $f$.
\begin{theorem}\label{thm:Cayley}
Let $f_0,\dots,f_\ell\in\mathbb{C}[\var_1^{\pm 1},\dots,\var_n^{\pm 1}]$ and set $T:={\rm Spec}\,\mathbb{C}[\beta_1^{\pm 1},\dots,\beta_\ell^{\pm 1}]$.
Then
\begin{equation*}
    \chi\left( V_{(\CC^*)^n}(f_1\cdots f_n) \right)
    \,=\,
    (-1)^\ell\cdot \chi\left( V_{(\mathbb{C}^*)^n\times T}\left(f\right)\right).
\end{equation*}
\end{theorem}
\begin{proof}
By an induction on $\ell$, the proof is reduced to the case that $\ell=1$.
Hence, let $\ell=1$ and denote $\mathcal{T} \!\coloneqq (\CC^*)^{n}\times T$.
We consider the fiber bundle $$\pi:\mathcal{T}\,\setminus \left( V_{\mathcal{T}}(f_0+\beta_1f_1)\cup V_{\mathcal{T}}(f_0f_1)\right)\to (\CC^*)^n\setminus V_{(\CC^*)^n}(f_0f_1)$$
defined by $\pi(\alpha,\beta_1)=\alpha$.
Any fiber $\pi^{-1}(\alpha)$ is a complex plane with two points removed.
By the multiplicative property of Euler characteristic, we obtain
\begin{equation}\label{eq:2.7}
\chi((\CC^*)^n\setminus V_{(\CC^*)^n}(f_0f_1))\,=\,-\chi\left( \mathcal{T}\,\setminus \left( V_{\mathcal{T}}(f_0+\beta_1f_1)\cup V_{\mathcal{T}}(f_0f_1)\right)\right).    
\end{equation}
The excision property of the Euler characteristic implies the following equalities:
\begin{align}
    \chi\left( \mathcal{T}\setminus  V_{\mathcal{T}}(f_0+\beta_1f_1)\right)
    \,=\,&\chi\left(  \mathcal{T}\setminus \left( V_{\mathcal{T}}(f_0+\beta_1f_1)\cup V_{\mathcal{T}}(f_0f_1)\right)\right) \nonumber\\
    &+
    \chi\left( V_{\mathcal{T}}(f_0f_1)\setminus V_{\mathcal{T}}(f_0+\beta_1f_1)\right);\label{eq:2.8}
\end{align}
and
\begin{align}
    \chi\left( V_{\mathcal{T}}(f_0f_1)\setminus V_{\mathcal{T}}(f_0+\beta_1f_1)\right)
    \,=\,&
    \chi\left( V_{\mathcal{T}}(f_0)\setminus V_{\mathcal{T}}(f_1)\right)\nonumber\\
    &+
    \chi\left( V_{\mathcal{T}}(f_1)\setminus V_{\mathcal{T}}(f_0)\right)\,=\,0. \label{eq:2.9}   
\end{align}
Here, the last equality follows from the equality $V_{\mathcal{T}}(f_1)\setminus V_{\mathcal{T}}(f_0)=V_{(\CC^*)^n}(f_1)\setminus V_{(\CC^*)^n}(f_0)\times T$ and $\chi(T)=0$.
In view of equations \eqref{eq:2.7},\eqref{eq:2.8} and \eqref{eq:2.9}, the theorem for $\ell=1$ follows.
\end{proof}

In general, for a morphism of algebraic varieties $\pi:F\to Z$, we define the Euler discriminant $\nabla_\chi^\pi(Z)$ by the formula \eqref{eq:2.3} where we replace $\chi_z$ by $|\chi(\pi^{-1}(z))|$. 
Using this notation, we can state the following corollary.

\begin{corollary}\label{cor:ED_morphism}
    Let $Z$ be a complex algebraic variety, and $f_0,\dots,f_\ell:(\mathbb{C}^*)^n\times Z\to\CC$ regular functions.
    Consider the projections $\pi_0:(\CC^*)^n\times Z\setminus V(f_0\cdots f_\ell)\to Z$ and $\pi:(\CC^*)^{n+\ell}\times Z\setminus V(f_0+\beta_1f_1+\cdots+ \beta_\ell f_\ell)\to Z$.
    Then, one has the equality $\nabla_\chi^{\pi_0}( Z)=\nabla_\chi^{\pi}( Z).$
\end{corollary}

\subsection{Principal A-determinant}\label{sec:PAD}
We begin by recalling a definition of principal $A$-determinant \cite[Chapter~10]{GKZbook}. We follow the notation of Section~\ref{sec:2_EulerDiscriminant}, except that we restrict to the case of a unique Laurent polynomial $f=f_0$ with monomial support $A=A_0\subset \mathbb{Z}^n$ such that the first entry of any $\uu \in A$ is $1$.
The $A$-discriminant variety $\nabla_{\!A}$ is defined as the closure of the image of the incidence variety
$$
\biggl\{ (\var,z)\in (\CC^*)^n\times \CC^A\mid \frac{\partial f}{\partial\var_i}(\var;z)=0,\ i=1,\dots,n\biggr\}
$$
through the natural projection $\CC^A\times (\CC^*)^n\to\CC^A$.
The $A$-discriminant variety is a proper homogeneous subvariety of $\CC^A$.
When $\nabla_{\! A}$ is a hypersurface, we say $A$ is {\it non-defective}, and the reduced defining polynomial of $\nabla_{\!A}$ is called the $A$-discriminant, denoted by $\Delta_A$.
If $\nabla_{\!A}$
is not a hypersurface, we set $\Delta_A=1$.

Let ${\rm Conv}(A)$ be the polytope obtained as the convex hull of the columns of $A$ in ${\mathbb{R}^n=\mathbb{R}\otimes_{\mathbb{Z}}\mathbb{Z}^n}$.
For a face $Q$ of $\Conv(A)$, we denote 
\begin{equation}\label{eq:face_poly}
f_Q(\var;z)\,:=\,\sum_{\uu\in A\cap Q}z_\uu\var^\uu,
\end{equation}
where $A\cap Q$ denotes the elements $u$ of $A$ such that $\uu \in Q$.

Next, we recall the notion of multiplicity necessary to define the principal A-determinant.
Let $S\subset \mathbb{Z}^n$ be an abelian semigroup admissible in the sense of \cite[Chapter~5, Definition~3.4]{GKZbook}.
We write $\Xi(S)$ for its group completion and $\Xi_{\mathbb{R}}$ for $\Xi(S)\otimes_{\mathbb{Z}}\mathbb{R}$.
Let $K(S)$ (resp. $K_+(S)$) denote the convex hull of $S$ (resp. $S\setminus\{0\}$) in $\Xi_{\mathbb{R}}$.
The subdiagram volume $u(S)$ is the normalized volume of the closure of the set $K(S)\setminus K_+(S)$, denoted $K_-(S)$, with respect to the volume form induced by $\Xi(S)$, see \cite[Chapter~3, \S 3.D]{GKZbook}:
\begin{equation}\label{eq:def_sub_vol}
    u(S)\,\coloneqq\,{\rm vol}_{\Xi(S)}(K_-(S)).
\end{equation}
For the trivial semigroup $S=\{0\}$, we set $u(S):=1$.
We write $S(A)$ for the subsemigroup of $\mathbb{Z}^n$ generated by $A$ and $0$.
In our setup, a face $Q$ of $\Conv(A)$ is identified with $S(Q)$ which is the subsemigroup of $S(A)$ generated by $A\cap Q$ and $0$. 

Given a face $Q$ of $\Conv (A)$, we define the number $i(Q,A)$ as the index
\begin{equation}
i(Q,A) \, \coloneqq\,[\mathbb{Z}^{n}\cap{\rm Lin}_{\mathbb{R}}(Q):{\rm Lin}_{\mathbb{Z}}(A\cap Q)].
\end{equation}
Here, for a subset $B\subset\mathbb{Z}^{n}$, ${\rm Lin}_{\mathbb{R}}(B)$ (resp. ${\rm Lin}_{\mathbb{Z}}(B)$) denotes the vector subspace spanned by $B$ in $\mathbb{R}^{n}=\mathbb{R}\otimes_{\mathbb{Z}}\mathbb{Z}^{n}$ (resp. $\mathbb{Z}^{n}$).
We define a semigroup $S(A)/Q$ as the subsemigroup of $\mathbb{Z}^{n}/\mathbb{Z}^{n}\cap{\rm Lin}_{\mathbb{R}}(Q)$ generated by $A$ and $0$.
The principal $A$-determinant is defined as the product over all faces $Q<\Conv(A)$:
 \begin{equation}\label{eq:PAD}
    E_{A}\,\coloneqq
    \,\prod_{\substack{Q<\Conv(A)\\ Q:\text{non-defective}}} \Delta_{A\cap Q}^{m_Q},  
\end{equation}
where the multiplicity $m_Q$ is given by
$
m_Q= i(Q,A)\cdot u(S(A)/Q).
$

\begin{proposition}\label{prop:degree}
The degree of the principal $A$-determinant equals $n\cdot {\rm vol}_{\ZZ}(\Conv(A)).
$
\end{proposition}

\begin{proof}
    Let us recall the definition of the principal $A$-determinant of $f$ as the $A$-resultant $R_A$ (\cite[Chapter~8, \S 2]{GKZbook}):
    \begin{equation}
E_A\,=\,R_A\biggl(f, \var_1\frac{\partial f}{\partial\var_1},\,\dots\,,\var_n\frac{\partial f}{\partial\var_n}\biggr) \,\in\, \CC[z_\uu\,;\,\uu\in A].
    \end{equation}
    The equivalence with \eqref{eq:PAD} is proved in \cite[Chapter~10, Theorem~1.2]{GKZbook}.
    For a non-zero complex number $a$, the principal $A$-determinant $E_A$ is multiplied by $a^{n\cdot {\rm vol}_{\ZZ}({\Conv(A))}}$ by replacing $f$ by $a\cdot f$ (\cite[Corollary 2.2, Chapter 8]{GKZbook}).
    This proves the proposition.
\end{proof}

The multiplicity $m_Q$ can be interpreted in terms of the difference of Euler characteristics of very affine varieties \cite[Theorem~2.36]{esterov2013discriminant}.
Although the following theorem is essentially known in the literature,  we present a proof that uses Kashiwara's local index theorem \cite{kashiwara1973index}.
\begin{theorem}\label{thm:drop_Euler}
   For $Q$ a non-defective face of $\Conv(A)$, the multiplicity $m_Q$ is given by
    $$
    \chi_{z^*}-\chi_{z_Q},
    $$
    where $z^*\in\CC^A\setminus \{E_A =0\}$ and $z_Q\in \{z\in\CC^A\mid\Delta_{A\cap Q}(z)=0\}$ are generic points.
\end{theorem}

\begin{proof}
    Firstly, note that the characteristic cycle for the GKZ system $M_A(c)$ with generic parameter $c$ is given by 
    $$m_0T^*_{\CC^A}\CC^A\,+\,
    \sum_{Q<\Conv(A)}m_Q T_{\nabla_{\!A\cap Q}}^*\CC^A,
    $$
    where $T^*$ denotes the conormal bundle of $\nabla_{\! A\cap Q}$.
    This result is stated in \cite[6.3]{loeser1991polytopes} and its proof can be read off from the proof of \cite[Theorem~4.6]{gelfand1990generalized}.
    In view of Kashiwara's local index formula \cite{kashiwara1973index} (see also \cite[Theorem~4.3.25]{dimca2004sheaves}), it follows that $\chi_{z_Q}=m_0-m_Q.$ This concludes the proof since $m_0=\chi_{z^*}$.
\end{proof}

Let $\vol(A)$ denote the lattice volume of the polytope $\Conv(A)$. The proof of the following theorem is available in \cite[Theorem~13]{amendola2019maximum}, establishing a connection between the vanishing locus of the principal $A$-determinant and the decrease in the Euler characteristic of $X_z$. 
\begin{theorem}
    The sigend Euler characteristic $(-1)^n\cdot \chi(X_z)$ is equal to  $\vol(A)$ if and only if $z\in \C^A\setminus V_{\CC^A}(E_A)$. Furthermore, when $E_A(z)=0$, we have $(-1)^n\cdot \chi(X_z) < vol(A).$
\end{theorem}
Note that, when the subvariety $Z$ as in \eqref{eq:2.3} is the entire parameter space, we have $\nabla_{\chi}(\CC^A) = V_{\CC^A}(E_A)$.
The following sections will concentrate on the case where the very affine variety $X_z$ is the complement of hyperplanes. We will present formulas for the Euler discriminant and the principal $A$-determinant.

\section{Principal A-determinant of a sparse arrangement}\label{sec:PAD_sparse}

\subsection{Sparse hyperplane arrangements}\label{sec:3.1}

Any hyperplane arrangement in $\mathbb{P}^k={\rm Proj}\CC[\var_0,\dots,\var_k]$ with $n+1$ hyperplanes is specified by a point in the Grassmannian $Gr(k+1,n+1)$ which is realized as a quotient of the space of $(k+1)\times(n+1)$ matrices of rank $k+1$ by the left multiplication by ${\rm GL}(k+1,\CC)$.
Namely, a point $z\in {\rm GL}(k+1,\CC)$ represented by a $(k+1)\times(n+1)$ matrix $(z_{ij})_{\substack{i=0,\dots,k\\ j=0,\dots,n}}$ defines a collection of linear forms $\{ h_j(\var;z):=z_{0j}\var_0+\cdots+z_{kj}\var_k\}_{j=0}^n$, whose vanishing locus defines a hyperplane arrangement 
\begin{equation}\label{eq:Az}
    \mathscr{A}_z \,\coloneqq\,\bigcup_{j=0}^n H_j,
\end{equation} 
where $H_j:=\{\var\in\mathbb{P}^k\mid h_j(\var;z)=0\}$.
Note that $H_j$ can be the whole projective space $\mathbb{P}^k$ which we do not exclude.
In this paper, we work on an open affine chart $U$ of ${\rm Gr}(k+1,n+1)$ where the matrix consisting of the first $k+1$ columns of the matrix is invertible.
By ${\rm GL}(k+1,\CC)$ action, any point in $U$ is given by a $(k+1)\times(n+1)$ matrix 
\begin{equation}\label{eq:C(z)}
\left[\,\,
E_{k+1}\,\mid\, z\,\,
\right]
\ \ \text{where}\ \ 
z\,=\,
\left[
\begin{array}{ccc}
z_{0k+1}&\cdots&z_{0n}\\
\vdots&\ddots&\vdots\\
z_{kk+1}&\cdots&z_{kn}
\end{array}
\right],
\end{equation}
and $E_{k+1}$ stands for the unit matrix of size $k+1$.
We identify $U$ with the Euclidean space $\CC^{(k+1)(n-k)}$. Consistently with Section~\ref{sec:2_EulerDiscriminant}, for any matrix $z\in\CC^{(k+1)\times(n-k)}$, we set 
\begin{equation}\label{eq:Xz_hyperplanes}
X_z\,\coloneqq\,\mathbb{P}^k\setminus\mathscr{A}_z.  
\end{equation}
We are interested in the study of its Euler discriminant $\nabla_\chi( Z)$. 
By \cite[Theorem 3.3.7]{orlikarrangements}, the signed Euler characteristic of $X_z$ is given by a combinatorial invariant called the {\it beta invariant}.
The beta invariant is also equal to the number of bounded regions when the hyperplane arrangement is real (\cite[Theorem 3.3.9]{orlikarrangements}).
Therefore, the Euler discriminant in such a case describes the locus for which some of the bounded chambers shrink.  

When $ Z = \CC^{(k+1)(n-k)}$, the Euler discriminant coincides with the vanishing locus of a principal $A$-determinant which was described in \cite[Chapter 9, \S1]{GKZbook}:
\begin{equation}\label{eq:EA_dense}
    E_A(z) \, = \, \prod_{I,J:|I|=|J|} \det(z_{I,J}),
\end{equation}
where the product runs over all subsets $I$, and $J$, of the set of rows, respectively columns, of the matrix $z$ from \eqref{eq:C(z)}.

We will focus on the case of a general subvariety $ Z$ in Section~\ref{sec:Euler_discriminant}.
In this section, we explore the case when $Z$ is an intersection of coordinate hyperplanes.
Also in this scenario, the Euler discriminant is given by the vanishing locus of a principal $A$-determinant.

Given a subspace $ Z$ as above, we define a bipartite graph $G$ with vertex set $V(G)=\{ 0,1,\dots,n\}$ and edge set $E(G)=\{ ij\mid z_{ij}\neq 0\}$.
In what follows, we write $z_G$ for a point $z\in Z$ to stress that the non-zero entries are specified by the edge set of the graph $G$.
In the context of graph theory, $z_G$ is called the Edmonds matrix of $G$ \cite[12.8.3]{tucker2004computer}.
We assume that any vertex of $G$ is an endpoint of an edge.
For a subset $T$ of the vertex set $V(G)$, the symbol $G_T$ denotes the subgraph of $G$ induced by $T$.
More precisely, $G_{T}$ denotes a subgraph obtained by deleting all the vertices in $V(G)\setminus T$ and the edges connected to them.
The \textit{graph neighborhood} $N(G;T)$ is the set of vertices in $V(G)$ which are connected to an element of $T$ by an edge of $G$. Let us set $V_1:=\{ 0,1,\dots,k\}$ (the left vertex set) and $V_2:=\{ k+1,\dots,n\}$ (the right vertex set).

By assigning a variable to each vertex $i\in V(G)$, we consider the vector space $\mathbb{R}^{n+1}$ and write ${e}_i$ for its $i$-th standard unit vector.
Here by convention, we note that $e_0=(1,0,\dots,0)$.
Consider the set 
\begin{equation}\label{eq:A_G}
    A_G\,:=\,\{ a_{ij}:=e_{i}+e_{j}\mid ij\in E(G)\}.
\end{equation}
Then, the convex hull of $A_G$ is called an \textit{edge polytope} of $G$ and denoted by $P_G$.

{We assume that $G$ is a connected graph.}
Otherwise, $G$ is decomposed into a disjoint union $G_1\sqcup\cdots\sqcup G_\ell$ with each $G_i$ connected.
The associated decomposition of vertices $V(G)=V(G_1)\sqcup\cdots\sqcup V(G_\ell)$ induces a direct product decomposition $\mathbb{R}^{n+1}=\mathbb{R}^{V(G_1)}\times\cdots\times\mathbb{R}^{V(G_\ell)}$. In this case, the edge polytope of $G$ is the direct product of those of $G_1,\dots,G_\ell$:
\begin{equation}
    P_G\,=\,P_{G_1}\times\cdots\times P_{G_\ell}.
\end{equation}

For a comprehensive introduction to polytopes, we refer to \cite{ziegler2012lectures}. For a polytope $P\subset \RR^{n+1}$, a \text{face} $F\subset P$ is the set of points in the polytope $P$ that maximizes a linear functional $\phi : \RR^{n+1} \to \RR$. The faces of dimension $\dim(P)-1$ are called \textit{facets}, where the dimension of a face is the dimension of the affine space spanned by its points. 

\begin{remark}\label{rmk:PG_NewtonPolytope}
   It is clear that the edge polytope $P_G$ is  the convex hull of the monomial support of the polynomial arising from the Cayley configuration of the monomial supports of the linear forms $h_j$, see \eqref{eq:Cayley_poly}. More explicitly, 
consider the polynomial ring $\CC[\var]$ where $\var = (\var_0,\dots,\var_{k},\var_{k+1},\dots,\var_n)$, then the polynomial associated to $A_G$ is
\begin{equation}\label{eq:f_G}
       f_G (\var;z) \, \coloneqq \, \sum_{ij \in E(G)} z_{ij} \var_i \var_j. 
\end{equation}
When setting $\var_{k+1}=1$ we get precisely the polynomial \eqref{eq:Cayley_poly} in the context of this section. Corollary~\ref{cor:ED_morphism} then guarantees that to study the Euler discriminant of $X_z$ we can look at the principal $A$-determinant for the matrix $A_G$.
\end{remark}

\subsection{Face structure of edge polytope of a bipartite graph}
Ohsugi and Hibi \cite{ohsugi1998normal} gave an explicit description of the facet structure of the edge polytope $P_G$ associated to a bipartite graph $G$, and also investigated its dimension.
The following result resumes the facet structure of $P_G$, as stated in \cite[Theorem~1.7]{ohsugi1998normal}:

\begin{theorem}[\cite{ohsugi1998normal}]\label{thm:HibiOhsugi}
    Let $Q$ be a facet of $P_G$.
    Then, $Q$ is one of the following types:
    \begin{itemize}
        \item[(a)] $Q$ is associated to an {\it ordinary} vertex $i_0$, i.e., a vertex such that $G_{V(G)\setminus \{ i_0\}}$ is connected. In this case, $F$ is given by the convex hull of a set $\{ a_{ij}\mid i,j\neq i_0\}$.
        \item[(b)] $F$ is associated to an {\it acceptable} subset $\varnothing\neq T\subset V_1$, i.e., a subset such that both $G_{T\cup N(G;T)}$ and $G_{V(G)\setminus(T\cup N(G;T))}$ are connected and the latter one is {\it non-trivial} in the sense that it has at least one edge. In this case, $F$ is given by the convex hull of a set $\{ a_{ij}\mid ij\in E(G_{T\cup N(G;T)})\cup E(G_{V(G)\setminus(T\cup N(G;T))})\}$.
    \end{itemize}
\end{theorem}

Thus, any facet of $P_G$ is of the form $P_{G_{I\cup J}}$ for some $I\subset V_1$ and $J\subset V_2$, or $P_{G_{I_1\cup J_1}}\times P_{G_{I_2\cup J_2}}$ for some $I_1,I_2\subset V_1$ and $J_1,J_2\subset V_2$.
In each case, the graphs $G_{I\cup J},G_{I_1\cup J_1}$ and $G_{I_2\cup J_2}$ are all connected.
Note that a face of the product $P_{G_{I_1\cup J_1}}\times P_{G_{I_2\cup J_2}}$ is a product of faces of $P_{G_{I_1\cup J_1}}$ and $P_{G_{I_2\cup J_2}}$.
In the next two results, we characterize the facet structure of an edge polytope. In particular, the following lemma is an immediate consequence of consecutively considering facets of a facet of $P_G$ and applying Theorem~\ref{thm:HibiOhsugi}.
\begin{lemma}\label{lem:3.3}
    For any face $Q$ of $P_G$, there exist disjoint subsets $I_1,\dots,I_\ell\subset V_1$ and $J_1,\dots,J_\ell\subset V_2$ such that
    \begin{enumerate}
    \item $G_{I_i\cup J_i}$ is connected for any $i=1,\dots,\ell$    
    \item $Q\,=\,P_{G_{I_1\cup J_1}}\times\cdots\times P_{G_{I_\ell\cup J_\ell}}$
    \end{enumerate} 
\end{lemma}

\noindent Conversely, the next lemma shows that any pair $(I,J)\in V_1\times V_2$ gives rise to a face of~$P_G$.
We denote $Q_{I,J}$ the convex hull of the set $\{ a_{ij}\mid ij\in E(G_{I\cup J})\}$.
\begin{lemma}
    Let $\varnothing\neq I\subset V_1$ and $\varnothing\neq J\subset V_2$.
    Then, $Q_{I,J}$ is a face of the polytope $P_G$.
\end{lemma}

\begin{proof}
We may assume that $\{ a_{ij}\mid ij\in E(G_{I\cup J})\}$ is non-empty.
    Let $\phi_i:\mathbb{R}^{n+1}\to\mathbb{R}$ be the $i$-th coordinate projection.
    Then, it is easily seen that the linear functional that maximize the face $Q_{I,J}$ is given by $\phi_{I,J}:=\sum\limits_{i\in I}\phi_i+\sum\limits_{j\in J}\phi_j$ .
\end{proof}

\begin{example}[$n=6, k=2$]\label{ex:artificial_pt1}
Let $G$ be the bipartite graph  with $V(G) = \{0,1,\dots,6\}$, and $E(G)=\{03,04,13,15,16,24,25,26\}$. This graph will serve as running example for the results in the rest of the paper. 
\begin{figure}[H]
\vspace{-0.5cm}
\centering
\begin{subfigure}[T]{0.40\linewidth}
%\centering
\vspace{1.3cm}
\includegraphics[]{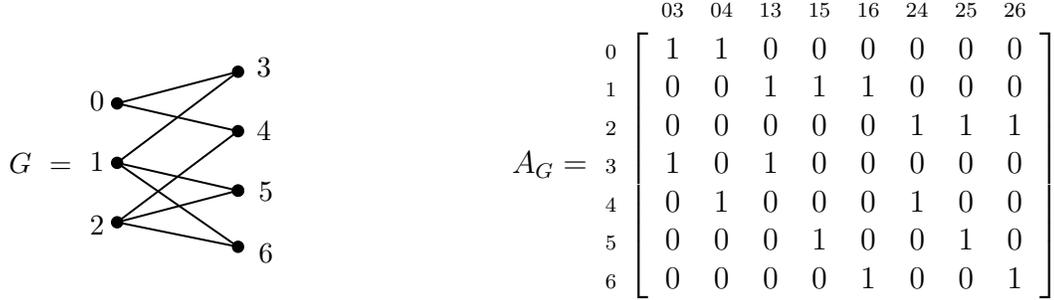}
\end{subfigure}
\begin{subfigure}[T]{0.40\linewidth}
$$
  A_G=
  \kbordermatrix{
      &03&04&13&15&16&24&25&26\\
     0&1 &1 & 0 & 0 & 0 & 0 & 0 &0\\
     1& 0 & 0 &1 &1 &1 & 0 & 0 &0\\
     2& 0 & 0 & 0 & 0 & 0 &1 &1 &1\\
     3&1 & 0 &1 & 0 & 0 &  0& 0 &0\\
     4& 0 &1 & 0 & 0 & 0 &1 & 0 &0\\
     5& 0 & 0 &  0&1 & 0 & 0 &1 &0\\
     6& 0 & 0 & 0 & 0 &1 & 0 & 0 &1
  }
  $$
\end{subfigure}
\caption{The bipartite graph $G$ and the matrix whose columns give the set $A_G$ from \eqref{eq:A_G}.}
\label{fig:graph_AG_artificial}
\end{figure} 
The edge polytope $P_G\subset\RR^6$ is the $5$-dimensional polytope obtained as the convex hull of the columns of the matrix $A_G$ in Figure~\ref{fig:graph_AG_artificial}. 
We computed using \texttt{Oscar.jl} \cite{OSCAR} that the $f$-vector is $(8, 26, 41, 31, 10)$ and the normalized volume $\text{vol}(P_G)$ equals $5$.
One can check that a facet of such a polytope is given by $P_{G_{I_1\cup J_1}}$ for $I_1 = \{0,1,2\}$ and $J_1=\{3,4,6\}$, while one of the codimension-2 faces is  $P_{G_{I_2\cup J_2}}$ for $I_2 = \{0,1,2\}$ and $J_2=\{3,4\}$.
\end{example}

\subsection{Principal A-determinant of a sparse arrangement}

In what follows, we will use the previous results on the face structure of an edge polytope $P_G$ to derive an explicit formula for the corresponding principal $A$-determinant, namely the one for the polynomial $f_G(\var;z)$ defined in Remark~\ref{rmk:PG_NewtonPolytope}.

A subset $m$ of $E(G)$ is called a {\it matching} if no pair of edges in $m$ is adjacent, i.e., they share a common vertex. We say that $m$ is a {\it perfect matching}, if it covers the whole set $V(G)$. While, we say that $m$ {\it saturates} $V_1$ if it covers the whole $V_1$.
The following characterization of the existence of a perfect matching is due to Hall (\cite[Theorem 5.2]{BM}):
\begin{theorem}[\cite{BM}]\label{thm:Hall}
    Suppose that $|V_1|\leq |V_2|$.
    Then, the graph $G$ has a matching which saturates $V_1$ if and only if, for any subset $W\subset V_1$, it holds that $|W|\leq |N(G;W)|$.
  \end{theorem}
\noindent In view of Theorem \ref{thm:Hall}, let us assume $|V_1|=|V_2|$. We define the following condition on~$G$:
\begin{itemize}
\item[$(*)$] $|V_1|=1$ and $G$ is connected, or $|V_1|>1$ and for any non-empty, proper subset $W$ of $V_1$, the inequality $|W|<|N(G;W)|$ holds.
\end{itemize}
\noindent
Note that, there is a clear relation between the hypotesis on the matching on the graph and the determinant of the Edmonds matrix $z_G$. 
In fact, it is readily seen that $G$ has a perfect matching if and only if the determinant $\det (z_{G})$ is not identically zero. Furthermore, condition $(*)$ corresponds to the irreducibility of the polynomial $\det(z_G)$:

\begin{proposition}\label{prop:irreducibility}
    Let $G$ be a bipartite graph with $|V_1|=|V_2|$.
    Then, $\det (z_G)$ is a non-zero irreducible polynomial in the variables $z_{ij}$ over any field if and only if $G$ satisfies \hbox{condition $(*)$.}  
\end{proposition}
  
\begin{proof}
    Suppose that $G$ does not satisfy $(*)$.
    When $|V_1|=1$, it is clear that $\det (z_G)$ is zero since $G$ is disconnected, hence there is no edge.
    Suppose that $|V_1|\geq 2$.
    By our assumption, we can find a non-empty proper subset $W$ of $V_1$ such that $|W|=|N(G;W)|$.
    By the definition of graph neighborhood, the matrix $z_G$ takes the following block-triangular form:
    \begin{equation}\label{eq:block triangulation}
  \begin{matrix}
\includegraphics[]{figures/zG_proof.tikz}
\end{matrix}          
    \end{equation}
    where $M \coloneqq N(G;W)$. It follows that $\det (z_G)$ is either reducible or zero.
    The argument above shows the necessary condition.
  
    Let us prove that the condition $(*)$ implies that $\det( z_G)$ is non-zero and irreducible.
    First of all, let $e=ij\in E(G)$ and write $G^\prime$ for the subgraph of $G$ obtained by removing the vertices on $e$.
    Let $i$ be the vertex on $e$ which lies on $V_1$.
    For any subset $W$ of $V_1\setminus \{ i\}$, we have $|N(G^\prime;W)|\geq |N(G;W)|-1$. Therefore, condition $(*)$ implies that the graph $G^\prime$ verifies the hypothesis of Theorem~\ref{thm:Hall}.
    By cofactor expansion, we obtain that 
    \begin{equation}
        \det (z_G)\,=\,\pm z_e\cdot \det( z_{G^\prime}) +(\text{terms that do not contain }z_{ij}).
    \end{equation}
    Since $\det (z_{G^\prime})\neq 0$, we proved that $z_{ij}$ appears in $\det (z_G)$ for any edge $e=ij$ of $G$.
    Now, $G$ having a perfect matching is equivalent to each connected component of $G$ having a perfect matching.
    If $G$ has a connected component $H$ that is strictly smaller than $G$, we obtain $|N(G;V_1\cap V(H))|=|V_1\cap V(H)|$.
    This contradicts the condition $(*)$.
    Hence, $G$ is connected.
    Let $\det (z_G)=p(z)\cdot q(z)$ be a decomposition.
    Fixing an edge $e=ij$ of $G$ we may assume that $z_{ij}$ appears in $p(z)$.
    By the definition of determinant, any element $z_{i^\prime j}$ $(i^\prime j\in E(G))$ or $z_{ij^\prime}$ $(ij^\prime\in E(G))$ must appear in $p(z)$.
    The connectedness of $G$ proves that any variable $z_{ij}$ $(ij\in E(G))$ belongs to $p(z)$, hence $q$ is a scalar.
  \end{proof}

Before presenting the main theorem of this section, we must first introduce the following lemma, which is established using linear algebra.
\begin{lemma}\label{lem:linear algebra}
    For a connected bipartite graph $G$, one has an identity
    \begin{equation}\label{eq:3.10}
        \mathbb{Z}^{n+1}\cap{\rm Ker}_{\RR}\phi\,=\,\sum_{ij\in E(G)}\ZZ\, \cdot\, a_{ij},
    \end{equation}
    where $\phi:=\phi_0+\cdots+\phi_k-\phi_{k+1}-\cdots-\phi_{n}$ and $a_{ij}$ as in \eqref{eq:A_G}.
\end{lemma}

\begin{proof}
    By definition, the right-hand side of \eqref{eq:3.10} is contained in the left-hand side.
    On the other hand, for any $0\leq i\leq k$ and $k+1\leq j\leq n$, a vector $e_i+e_j$ belongs to the right-hand side of \eqref{eq:3.10}.
    In fact, there is a sequence $i_1j_1, i_2j_1, i_2j_2,\dots,i_\ell j_\ell\in E(G)$ with $i_1=i$ and $j_\ell=j$ since $G$ is connected.
    It follows that $e_i+e_j=a_{i_1j_1}-a_{i_2j_1}+a_{i_2j_2}-\dots+a_{i_\ell j_\ell}$.
    Since any element of the right-hand side of \eqref{eq:3.10} can be written as a $\mathbb{Z}$-linear combination of vectors $e_i+e_j$, the lemma follows.
\end{proof}

\noindent We can finally state the main result of this section.

\begin{theorem}\label{thm:PAD}
Let $G$ be a connected bipartite graph. Then, one has a formula
    \begin{equation}\label{eq:product of determinants}
    E_{A_G}(z)\,=\,\prod_{\substack{I,J:
    |I|=|J|,\\
    \text{ $G_{I\cup J}$ is connected and }(*)}} \det (z_{I,J})^{u(S(A_G)/Q_{I,J})},
    \end{equation}
where the product runs over all non-empty subsets $I\subset V_1$ and $J\subset V_2$, and $z_{I,J}$ denotes the submatrix of $z_G$ obtained by selecting the rows in $I$ and the columns in $J$.
\end{theorem}

\begin{proof}
To simplify the notation, we write $A$ for $A_G$ below.
Let $Q$ be a face of $P_G$ and let $H$ be the corresponding subgraph of $G$, i.e., $P_H = Q$ as in Lemma \ref{lem:3.3}. 
We denote $f = f_G(\var;z)$ as in \eqref{eq:f_G}. We divide the proof into three main cases.

\begin{enumerate}
\item[(i)]\underline{The subgraph $H$ has several connected components:}
let $H_1=G_{I_1\cup J_1},\dots, H_\ell=G_{I_\ell\cup J_\ell}$ be its connected components $(\ell\geq 2)$.
The face polynomial, defined in \eqref{eq:face_poly}, is given by $f_Q=\sum_{a=1}^\ell\left( \sum_{i\in I_a,j\in J_a}z_{ij}\var_i\var_j\right)$.
Let us denote $\var_{I_a} = (\var_i)_{i\in I_a}$ and $\var_{J_a} = (\var_i)_{i\in J_a}$ for $a=1,\dots,\ell$. Then, the vanishing locus of $(A\cap Q)$-discriminant is given by the Zariski closure of the set:
\[
\nabla_{\!A\cap Q}=\bigcap_{a=1}^\ell\bigl\{ z\mid \exists (\var_{I_a},\var_{J_a}) \in (\CC^*)^{I_a}\times(\CC^*)^{J_a} \,\,: \,\,  \var_{I_a} \cdot z_{I_a,J_a}=   z_{I_a,J_a}\cdot \var_{J_a}^T=0\bigr\} .
\]
Therefore, by definition, $\nabla_{\!A\cap Q}$ is the union of proper varieties.
It follows that the variety $\overline{\nabla_{\!A\cap Q}}$ has codimension higher than one.
Thus, the $(A\cap Q)$-discriminant is $1$.

\item[(ii)]\underline{The subgraph $H$ is connected and of the form $H=G_{I\cup J}$ with $|I|<|J|$:}
\hbox{we distinguish two sub-cases.}

\underline{(ii-a) We assume there exists a matching $m$ which saturates $I$:} let us take a generic point $z$ from the $(A\cap Q)$-discriminant variety.
By definition, there exists a row vector $\var_I \in(\mathbb{C}^*)^I$ such that $\var_I \cdot z_{I,J}=0$, which means that the set of row vectors of $z_{I,J}$ is linearly dependent.
Therefore, for any subset $J_1\subset J$ with $|J_1|=|I|$, the minor $\det( z_{I,J_1})$ must vanish at $z$.
Let $J_1\subset J$ be the subset consisting of $j\in J$ that appear in $m$, then
$\det (z_{I,J_1})$ is not a zero polynomial.
On the other hand, there must exist a vertex $i_1\in I$ which is connected to $j_2\in J\setminus J_1$ since $G_{I\cup J}$ is connected.
We write $j_1\in J_1$ for the vertex paired with $i_1$ in $m$.
We define the new matching $m':=(m\setminus\{ (i_1,j_1)\})\cup\{ (i_1,j_2)\}$ and set $J_2:=(J_1\setminus \{ j_1\})\cup\{ j_2\}$.
We obtain that $\det (z_{I,J_2})$ is not a zero polynomial.
Summing up the argument above, the $(A\cap Q)$-discriminant variety is contained in the intersection $\{\det (z_{I,J_1})=0\}\cap\{\det (z_{I,J_2})=0\}$.
%Since $\det z_{I,J_1}$ is not divisible by $\det z_{I,J_2}$, 
We conclude that the $(A\cap Q)$-discriminant variety has codimension higher than one.

\underline{(ii-b) We assume that there there is no matching $m$ that saturates $I$:}
by definition of the $(A\cap Q)$-discriminant, there exists a vector $\var_I \in(\mathbb{C}^*)^I$ such that $\var_I \cdot z_{I,J}=0$.
Theorem~\ref{thm:Hall} and the connectedness of $G_{I\cup J}$ imply that there exists a non-empty proper subset $\varnothing\neq W\subsetneq I$ such that $|N(G_{I\cup J};W)|<|W|$.
It follows that 
$$\var_{(I\setminus W)}\cdot z_{I\setminus W,J\setminus N(G_{I\cup J};W)}=0.$$
Since $|I\setminus W|<|J\setminus N(G_{I\cup J};W)|$, there exist subsets $I'\subset I\setminus W$ and $J'\subset J\setminus N(G_{I\cup J};W)$ such that $|I'|<|J'|$ and $G_{I'\cup J'}$ is a connected component of $G_{(I\setminus W)\cup(J\setminus N(G_{I\cup J};W))}$.
It follows by construction that $\var_{(I')}\cdot z_{I',J'}=0$.
If there exists a matching $m$ of $G_{I'\cup J'}$ which saturates $I'$, the same argument as (ii-a) shows that the $(A\cap Q)$-discriminant has codimension higher than one.
If there is no such matching $m$, %of $G_{I'\cup J'}$ which saturates $I'$,
we repeat the argument above to find  smaller subsets $I'',J''$ such that $|I''|<|J''|$ and $G_{I''\cup J''}$ is connected.
Repeating this argument eventually yields the case (ii-a).

\item[(iii)]\underline{The subgraph $H$ is connected and of the form $H=G_{I\cup J}$ with $|I|=|J|$:}
\hbox{we distinguish two sub-cases.}

\underline{(iii-a) The graph $H = G_{I\cup J}$ does not satisfy $(*)$:}
in view of the proof of Proposition \ref{prop:irreducibility}, $ z_{I,J}$ is decomposed into a block triangular form as in \eqref{eq:block triangulation}.
If $z$ is a generic point of the $(A\cap Q)$-discriminant variety, there exists $\var_I\in(\mathbb{C}^*)^{|I|}$ and $\var_J\in(\mathbb{C}^*)^{|J|}$ such that $\var_I \cdot z_{I,J}=0$ and $z_{I,J}\cdot \var_J^T=0$.
Let $z_{I_1,J_1}$ be its upper-left block and let $\det( z_{I_2,J_2})$ be its lower-right block.
Then, we obtain $z_{I_1,J_1}\cdot \var_{J_1}^T=0$ and $\var_{I_2}\cdot z_{I_2,J_2}=0$.
This means that the $F$-discriminant variety is contained in $\{ \det (z_{I_1,J_1})=\det (z_{I_2,J_2})=0\}$.
Thus, the $A\cap Q$-discriminant is $1$.

\underline{(iii-b) The graph $H = G_{I\cup J}$ satisfies $(*)$:}
it is readily seen that the $A\cap Q$-discriminant variety is contained in $\{\det (z_{I,J})=0\}$.
Let us prove the other inclusion.
For any $i\in I$ and $j\in J$,  condition $(*)$, together with Theorem~\ref{thm:Hall}, implies that $G_{I\cup J\setminus\{ i,j\}}$ has a perfect matching.
This means that the $(i,j)$-cofactor $\Delta_{ij}$ of the matrix $z_{I,J}$ is a non-zero polynomial.
It follows from Laplace expansion of matrices that $\{\det (z_{I,J})=0\}\setminus\cup_{i,j}\{\Delta_{ij}=0\}$ is contained in the $A\cap Q$-discriminant variety.
Taking the Zariski closure, we conclude that the $A\cap Q$-discriminant is precisely given by $\det (z_{I,J})$.
\end{enumerate}

Formula \eqref{eq:product of determinants} follows since $i(Q_{I,J},A)=1$ by Lemma~\ref{lem:linear algebra}.
Note that $G_{I\cup J}$ is connected when the condition $(*)$ is satisfied. 
\end{proof}

Note that, in \cite[Proposition 4.4]{galashin2024trianguloids}, the secondary fan, i.e., the normal fan of the principal $A$-determinant, is determined as the normal fan of a polynomial $f$ which is different from \eqref{eq:product of determinants}.
The polynomial $f$ is defined as a product of the form \eqref{eq:product of determinants} where the indices $I,J$ run over pairs that satisfy the condition of Theorem \ref{thm:Hall}.
The reduced polynomial of $f$ is identical to that of \eqref{eq:product of determinants}.

\begin{example}[$n=6, k=2$]\label{ex:artificial_pt2}
Let $G$ be as in Example \ref{ex:artificial_pt1}. Figure~\ref{fig:hyperplanes} shows the matrix $z_G$ and the hyperplane arrangement associated to it.
\begin{figure}[ht]
  \begin{center}
\includegraphics[]{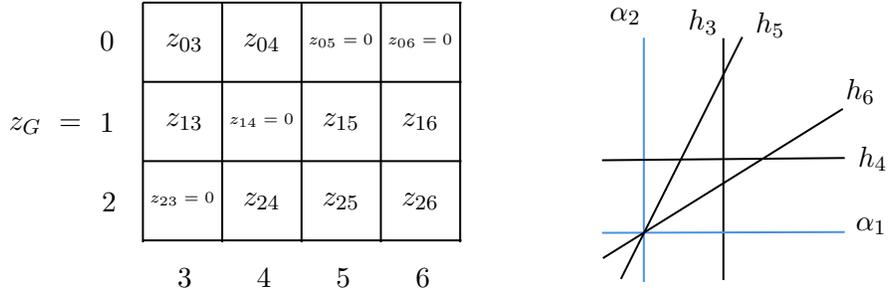}\caption{The matrix $z_G$ with $G$ as in Example~\ref{ex:artificial_pt1} and the induced hyperplane arrangement.}
\label{fig:hyperplanes} 
  \end{center}
  \end{figure}
Using Theorem~\ref{thm:Euler_discr} we compute the factors of the principal $A$-determinant $E_{A_G}(z_G)$:
\begin{equation}\label{eq:PAD_example}
     z_{03}^3 z_{13}^3 z_{04}^3 z_{24}^3 z_{15}^2 z_{25}^2 z_{16}^2 z_{26}^2 (z_{15}z_{26}-z_{16}z_{25})^2(z_{03}z_{24}z_{15} + z_{13}z_{04}z_{25})(z_{03}z_{24}z_{16} + z_{13}z_{04}z_{26}).
\end{equation}
This is a polynomial of degree $30 = 6\cdot 5$, consistently with Proposition~\ref{prop:degree}.  Exponents are computed by the method as in Example \ref{ex:artificial_pt3} in \S\ref{sec:subdiagram_volume}.
Clearly, the factors of the form $z_{ij}$  correspond to vertices of the polytope and therefore to minors determined by subsets $I= \{i\}$ and $J = \{j\}$. The three non-linear factors correspond respectively to the subgraphs displayed in Figure~\ref{fig:three subgraphs}.
\begin{figure}[h!]
  {\centering
\includegraphics[]{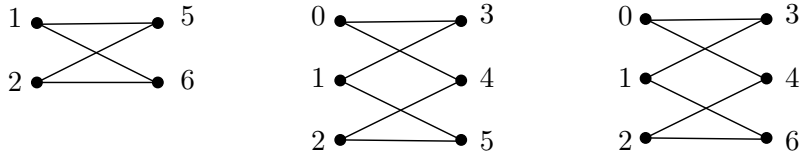}
\caption{Subgraphs $Q_{12,56}, Q_{012,345}, Q_{012,346}$.}\label{fig:three subgraphs}
 }
 \end{figure}
Choosing vertices, for instance, in the sets $I=\{0,1,2\}$ and $J=\{4,5,6\}$ leads to a non-connected graph and therefore to a reducible determinant of the associated Edmonds matrix, e.g.,  $z_{I,J}=z_{04}(z_{15}z_{26}-z_{16}z_{25})$. 
The choice $I=\{0,1\}$ and $J=\{5,6\}$ corresponds to a non-connected diagram which induces an identically vanishing determinant. Finally, the subgraph $G_{I\cup J}$ with $I=\{0,1\}$ and $J=\{5,6\}$ is an example of a connected graph which does not verify $(*)$, leading again to an identically zero determinant. 
\begin{figure}[h]
 {\centering
\includegraphics[]{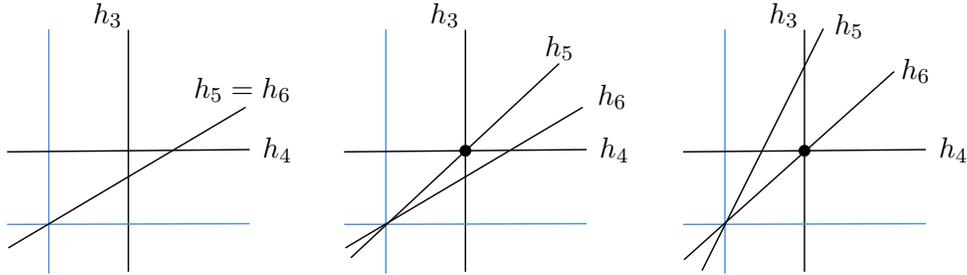}\caption{The degeneration of the generic hyperplane arrangement in Figure~\ref{fig:hyperplanes} to coefficient choices on the subspaces defined by the last three factors of $E_{A_G}(z_G)$ in \eqref{eq:PAD_example}.}\label{fig:hyperplanes_subspaces} }
  \end{figure}

Figure~\ref{fig:hyperplanes} (right) shows the hyperplane arrangement induced by a choice of coefficients $z_{ij}$ in the matrix $z_G$ that do not belong to the locus $\{E_{A_G}(z_G)\} = 0$. In Figure~\ref{fig:hyperplanes_subspaces} instead we show the hyperplane arrangements associated to choices of coefficients which lie in the vanishing locus of the principal $A$-determinant. 
\end{example}

\subsection{Subdiagram volume from bipartite graphs}\label{sec:subdiagram_volume}

Computing the subdiagram volume based on its definition, c.f. \eqref{eq:def_sub_vol}, is often nontrivial and may pose computational challenges. We present a method tailored to the case where the matrix $A$ arises from a bipartite graph $G$, exploiting the graph's combinatorial structure. This approach uncovers intriguing connections with non-homogeneous toric ideals.  

Let $G$ be a bipartite graph. By abuse of notation, we write $A_G:\ZZ^{E(G)}\to{\rm Ker}\phi\subset \ZZ^{n+1}=\ZZ^{V(G)}$ for the linear map which sends $e(ij)$ to $a_{ij}$ defined by \eqref{eq:A_G}, where $e(ij)$ denotes the unit vector corresponding to the edge $ij\in E(G)$, and $\phi$ as in Lemma~\ref{lem:linear algebra}.
The kernel of the map $A_G$ is naturally identified with the first homology group of the graph $G$:
\begin{equation}\label{eq:ses}
0 \longrightarrow H_1(G,\ZZ) \longrightarrow \ZZ^{E(G)} \xrightarrow{A_G}{\rm Ker}\phi \longrightarrow 0. 
\end{equation}
Indeed, a cycle $C:=\{i_1j_1, i_2j_1, i_2j_2, \dots,i_\ell j_\ell,i_1j_\ell\}\subset E(G)$ defines an element of the kernel of $A_G$ by taking the linear combination:
\begin{equation}\label{eq:kernel element}
    e(i_1j_1)-e(i_2j_1)+e(i_2j_2)-\cdots+e(i_\ell j_\ell)-e(i_1j_\ell).
\end{equation}
The homology space $H_1(G,\ZZ)$ is furthermore generated by the independent 
cycles of $G$.
Let $H$ be a connected subgraph of $G$ defined by a non-defective face $Q< P_G$.
Therefore, the subgraph $H$ also admits a short exact sequence as in \eqref{eq:ses}.
The quotient of these short exact sequences determines another short exact sequence:
\begin{equation}\label{eq:ses_complement}
0 \longrightarrow H_1(G,H;\ZZ) \longrightarrow \ZZ^{E(G)\setminus E(H)} \xrightarrow{A_{G/H}} \ZZ^{n+1}/\text{Lin}_{\ZZ} (Q) \longrightarrow 0.    
\end{equation}
The image of $\ZZ_{\geq 0}^{E(G)\setminus E(H)}$ defines the semigroup $S(A)/Q$. Our goal is to identify this short exact sequence with the one representing the homology of a specific graph. To achieve this, we denote $G/H$ the graph obtained from $G$ by contracting the edges in $E(H)$.
The kernel of the map $A_{G/H}$ can be identified with the homology group $H_1(G/H;\ZZ)$ via the identification of homology groups $H_1(G,H;\ZZ)\simeq H_1(G/H;\ZZ)$.
Note that an even cycle of $G/H$ is represented by a set $\{ i_1j_1,i_2j_1,i_2j_2,\dots,i_\ell j_\ell, i_{\ell+1}j_\ell\}\subset E(G)\setminus E(H)$ where $i_{1}$ and $i_{\ell+1}$ are identified in the quotient $G/H$.
Likewise, an odd cycle of $G/H$ is represented by $\{ i_1j_1,i_2j_1,i_2j_2,\dots,i_\ell j_\ell\}\subset E(G)\setminus E(H)$ where $i_{1}$ and $j_{\ell}$ are identified in the quotient $G/H$. Analogously to~\eqref{eq:kernel element}, every even cycle defines an element of ${\rm Ker}A_{G/H}$, whereas an odd cycle, represented as above, gives rise to a sum
\begin{equation*}
    e(i_1j_1)-e(i_2j_1)+e(i_2j_2)-\cdots+e(i_\ell j_\ell).
\end{equation*}
On the other hand, there is an identification $\ZZ^{V(G)\setminus V(H)}\simeq \ZZ^{n+1}/\text{Lin}_{\ZZ}(Q)$ which is induced by the natural embedding $\ZZ^{V(G)\setminus V(H)}\hookrightarrow\ZZ^{n+1}=\ZZ^{V(G)}$.

Summing up the arguments above, the short exact sequence in~\eqref{eq:ses_complement} can be identified with the canonical short exact sequence associated to the homology group of the graph $G/H$:
\begin{equation}\label{eq:ses_complement2}
0 \longrightarrow H_1(G/H;\ZZ) \longrightarrow \ZZ^{E(G/H)} \xrightarrow{A_{G/H}} \ZZ^{V(G)\setminus V(H)} \longrightarrow 0.    
\end{equation}

By construction, the matrix $A_{G/H}$ is obtained from $A_{G}$ by removing columns labeled by $E(H)$ and rows labeled by $V(H)$. We have proved the following 
\begin{proposition}
Let $G$ be a bipartite graph, and let $P_G$ denote its edge polytope. Let $A$ be the matrix whose columns are the vertices of $P_G$. Then for any non-defective face $Q$ of $P_G$, the following holds:
\begin{equation}
    u(S(A)/Q)\,=\, \text{vol}_{\ZZ}({\rm Conv}(\{ 0\}\cup A_{G/H}))-\text{vol}_{\ZZ}({\rm Conv}(A_{G/H})).
\end{equation}
\end{proposition}

\begin{example}\label{ex:artificial_pt3}  
    Let $G$ be as in Example \ref{ex:artificial_pt1}. Given a set $S\subset V(G)$, we denote $G_S$ the subgraph of $G$ which is obtained from $G$ by selecting the vertices which appear in $S$ and the edges attached to such vertices. Figure~\ref{fig:three_subgraphs} illustrates the quotients graphs $G/H$ for the choices of $H$ in $\{G_{03},G_{15},G_{1256}, G_{012345}, G_{012346}\}$. Note that any other choice of subdiagram $H$ corresponding to a single edge leads to one of the two topologies of graph in the first line of Figure~\ref{fig:three_subgraphs}.
    \begin{figure}[h]
  \begin{center}
\includegraphics[]{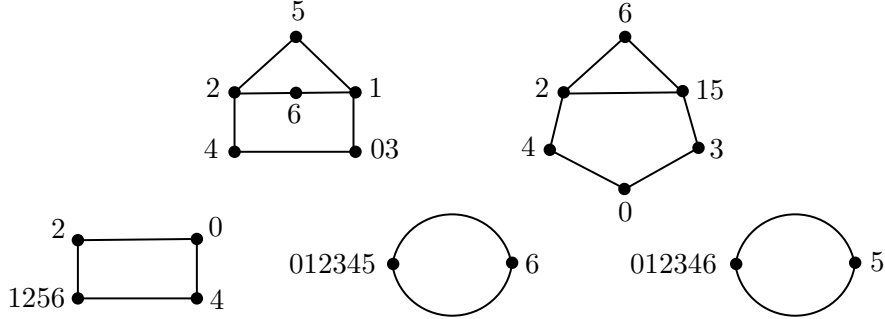}
  \end{center}
  \caption{Quotients graphs $G/H$ for $H=G_{03},G_{15},G_{1256}, G_{012345}, G_{012346}$.}\label{fig:three_subgraphs}
  \end{figure}
 We provide some of the matrices $A_{G/H}$, which can be derived from the matrix 
$A$ in Figure~\ref{fig:graph_AG_artificial} by deleting specific rows and edges, as described above:
  $$
  A_{G/G_{1256}}=
  \kbordermatrix{
    & 03 &04 &13 &24 \\
   0& 1  &1  &0  &0 \\
   3& 1  &0  &1  &0\\
   4& 0  &1  &0  &1
  },\ \ \ 
  A_{G/G_{012345}}=
  \kbordermatrix{
    & 16 &26 \\
   6& 1  &1  
  },\ \ \ 
  A_{G/G_{012346}}=
  \kbordermatrix{
    & 15 &25 \\
   5& 1  &1  
  }.
  $$
The matrices $A_{G/H}$ for $H=G_{03},G_{15}$ are obtained in the same manner. 
\end{example}

In addition to the description given above, the subdiagram volume has also an interpretation as the multiplicity at the point $0$ of a toric variety, see~\cite[Theorem~3.14]{GKZbook}. We conclude this subsection by exploiting this perspective in our case of interest.

Consider $Y_{G/H}$ to be be the affine toric variety associated to the linear map $A_{G/H}$.
In general, $Y_{G/H}$ is not homogeneous.
We write $I_{G/H}\subset\CC[z_{ij};ij\in E(G)\setminus E(H)]$ for its defining toric ideal.
The following lemma is a direct consequence of the fact that $H_1(G/H;\ZZ)$, as defined in \eqref{eq:ses_complement2}, is generated by cycles of $G/H$.
A practical consequence of this lemma is  that $I_{G/H}$ is obtained from the toric ideal $I_G\subset\CC[z_{ij};ij\in E(G)]$, associated to $G$, by substituting $z_{ij}=1$ for $ij\in E(Q)$.

\begin{lemma}\label{lemma:f_C}
    Given a cycle $C$ of $G/H$, we set $f_C$ to be
    \begin{equation*}
        \begin{cases}
            z_{i_1j_1}\cdots z_{i_\ell j\ell}-z_{i_2j_1}\cdots z_{i_\ell j_{\ell-1}}z_{i_{\ell+1}j_\ell}&(\scalebox{0.85}{$C=\{ i_1j_1,i_2j_1,i_2j_2,\dots,i_\ell j_\ell, i_{\ell+1}j_\ell\}\subset E(G)\setminus E(Q):\text{even}$})\\
            z_{i_1j_1}\cdots z_{i_\ell j\ell}-z_{i_2j_1}\cdots z_{i_\ell j_{\ell-1}}&(\scalebox{0.8}{$C=\{ i_1j_1,i_2j_1,i_2j_2,\dots,i_\ell j_\ell\}\subset E(G)\setminus E(Q):\text{odd}$}).
        \end{cases}
    \end{equation*}
    Then, the set $\{ f_C\mid C:\text{cycle of }G\}$ generates the ideal $I_{G/H}$.
\end{lemma}

Therefore, the generators obtained in the lemma can be used to compute the the subdiagram volume $u(S(A)/Q)$
 as the local multiplicity ${\rm mult}_0(Y_{G/H})$.
Let $R_{G/H}:=\CC[z_{ij};ij\in E(G)\setminus E(H)]/I_{G/H}$ and write $\mathfrak{m}\subset R_{G/H}$ for the maximal ideal corresponding to the origin.
Then, ${\rm mult}_0(Y_{G/H})$ can be read off from the leading coefficient of the Hilbert polynomial of the associated graded ring ${\rm gr}_{\mathfrak{m}}(R_{G/H})$ with respect to the $\mathfrak{m}$-adic filtration of $R_{G/H}$.
To compute the associated graded ring, a standard technique of Gr\"obner basis can be adapted.
For these claims, see, e.g., \cite[\S12.1 and Proposition 15.28]{eisenbud2013commutative}.

\begin{example}
Let $G$ be as in Example \ref{ex:artificial_pt1}.
For the subgraphs $H=G_{03},G_{15}$, the Hilbert polynomials of ${\rm gr}_{\mathfrak{m}}(R_{G/H})$ are given by
$$\frac{1}{8}t^{4}+\frac{11}{12}t^{3}+\frac{19}{8}t^{2}+\frac{31}{12}t+1,\ \ \ 
\frac{1}{12}t^{4}+\frac{2}{3}t^{3}+\frac{23}{12}t^{2}+\frac{7}{3}t+1.$$
The local multiplicity ${\rm mult}_0(Y_{G/H})$ is thus given by $\frac{1}{8}\times 4!=3$, and $\frac{1}{12}\times 4!=2$, respectively.
On the other hand, for subgraphs $H=G_{1256},G_{012345},G_{012346}$, the Hilbert polynomials are given by $t^2+2t+1$, $t+1$, and $t+1$.

For the subgraph $H=G_{03}$, the toric ideal $I_{G/H}$ is generated by the following elements corresponding to three cycles in $G/H$:
$$
z_{16}\, z_{25}-z_{15}\, z_{26},\:\
 z_{04}\, z_{13}\, z_{26}- z_{16}\, z_{24},\:\ z_{04}\, z_{13}\, z_{25}- z_{15}\, z_{24}.
$$
The lowest homogeneous parts of these polynomials $z_{16}\, z_{25}-z_{15}\, z_{26},  z_{16}\, z_{24}, z_{15}\, z_{24}$ generates an ideal $I'\subset \CC[z_{04},z_{13},z_{15},z_{16},z_{24},z_{25},z_{26}]$.
Using the algorithm \cite[Proposition 15.28]{eisenbud2013commutative}, we find that ${\rm gr}_{\mathfrak{m}}(R_{G/H})\simeq \CC[z_{04},z_{13},z_{15},z_{16},z_{24},z_{25},z_{26}]/I'$.
In several examples, we confirmed that the lowest homogeneous parts of binomials in Lemma~\eqref{lemma:f_C} generate an ideal $I'$ such that ${\rm gr}_{\mathfrak{m}}(R_{G/H})\simeq \CC[z_{ij};ij\in E(G)\setminus E(H)]/I'$, which seems to hold in general. In \eqref{mathrepo}, we provide the code in  \texttt{Macaulay2} \cite{M2} to perform such a computation.
\end{example}

\section{Determinants and Euler discriminants}\label{sec:Euler_discriminant}
In this section, we focus on the case when $Z$ is a subvariety of $\CC^{(k+1)(n-k)}$. Namely, we consider a matrix $z$ of size $(k+1)\times(n-k)$, as in~\eqref{eq:C(z)}, except that here $z$ varies in $Z$. Let us set $S:=\{ (I,J) \in \{ 0,\dots,k\}\times \{k+1,\dots,n\} \,:\, |I| = |J|, \,\det(z_{I,J})\neq 0\}$.
We define the reduced discriminant of $Z$ as
\begin{equation}\label{eq:discriminant}
E^{\,{\rm red}}_\chi (z)\,\coloneqq\,\prod_{(I,J)\in S}\det (z_{I,J}).    
\end{equation}
To emphasize the choice of $Z$, we may write $E^{\,{\rm red}}_\chi( Z)$ for $E^{\,{\rm red}}_\chi(z)$. Analogously to Section~\ref{sec:PAD_sparse}, 
for each $z\in Z$, the column vectors in \eqref{eq:C(z)} define a hyperplane arrangement $\mathscr{A}_z$ in $\mathbb{P}^k(\mathbb{C})={\rm Proj}\CC[\var_0,\dots,\var_k]$, and its complement $X_z$, see~\eqref{eq:Az} and~\eqref{eq:Xz_hyperplanes}, respectively.
\begin{theorem}\label{thm:Euler_discr}
Suppose that $\chi_z>0$ for some $z\in Z$. Then, the following identity is satisfied:
    \begin{equation}
        \nabla_\chi( Z)\,=\,\{ z\in  Z\mid E^{\,{\rm red}}_\chi(z)=0\}.
    \end{equation}
\end{theorem}

\begin{proof}
    By abuse of notation, We write $\mathscr{A}_z$ for the dehomogenized hyperplane arrangement that \eqref{eq:C(z)} determines by choosing $\{ \var_0=0\}$ as the hyperplane at infinity.
    Suppose that $z_1,z_2\in Z\setminus ({\rm R.H.S.})$.
    Then, the intersection posets are equivalent: $L(\mathscr{A}_{z_1})=L(\mathscr{A}_{z_2})$.
    This is because the matroids corresponding to the points $z_1$ and $z_2$ have the same basis set.
    Therefore, $\chi_z=\chi^*$ for any $z\notin ({\rm R.H.S.})$ since the complement of the $({\rm R.H.S.})$ has non-empty intersection with the locus where the Euler characteristics takes its maximum.
    This proves the inclusion $({\rm L.H.S.})\subset({\rm R.H.S.})$.

    Let us show the other inclusion.
    We write $M_z$ for the matroid of arrangements defined by the matrix \eqref{eq:C(z)}.
    As long as $E^{\,{\rm red}}_\chi(z)\neq 0$, the matroid $M_z$ does not depend on the choice of $z$.
    We denote this matroid by $M^*$. Let $z\in Z$ such that $E^{\,{\rm red}}_\chi(z)=0$.
    We denote $A_z=A(M_z)$ the Orlik-Solomon algebra of $M_z$. To describe this algebra, we introduce the following notation. For an ordered subset $T=\{ t_1,\dots t_p\}\subset\{ 1,\dots,n\}$, we define
    \begin{equation*}
        e_T\,:=\,e_{t_1}\cdots e_{t_p}\ \ \ \text{and}\ \ \ \partial e_T\,:=\,\sum_{k=1}^{p}(-1)^{k-1}e_{t_1}\cdots \widehat{e_{t_k}}\cdots e_{t_p},
    \end{equation*}
    where $\widehat{e_{t_k}}$ means that the element $e_{t_k}$ is omitted. The Orlik-Solomon algebra $A_z$ is then given as the quotient of the exterior algebra $\wedge^\bullet \mathbb{C}^{\{1,\dots,n\}}$ by the ideal $I(\mathscr{A}_z)$ generated by the set
    \begin{equation*}
        \left\{ e_T\mid \bigcap_{t\in T}H_t=\varnothing\right\}
        \cup
        \left\{
        \partial e_T\mid \bigcap_{t\in T}H_t\neq \varnothing\ \text{and}\ T\text{ is dependent}
        \right\}.
    \end{equation*}
    The $i$-th graded part of $A_z$ is denoted by $A_z^i$.
    Since any independent subset of $M_z$ is that of $M^*$, there is a natural surjection
    \begin{equation}\label{eq:surjection}
        A(M^*)\to A_z
    \end{equation}
    which respects the grading. We consider generic parameters $\lambda_1,\dots,\lambda_n\in\mathbb{C}$ and set
    \begin{equation*}
    \omega_\lambda:=\lambda_1e_1+\cdots+\lambda_ne_n,
    \end{equation*}
    where $e_i$ is the element of the Orlik-Solomon algebra (either $A(M^*)$ or $A_z$) which corresponds to $i$-th hyperplane.
    The surjection \eqref{eq:surjection} induces a surjection
    \begin{equation}\label{eq:degeneration}
        H^k(M^*)\,:=\,A^k(M^*)/\omega_\lambda A^{k-1}(M^*)\,\overset{\varphi}{\to}\, A^k_z/\omega_\lambda A^{k-1}_z \,=:\,H^k(M_z).
    \end{equation}
    By the results of \cite{esnault1992cohomology} (see also \cite{schechtman1994local}), we have $\dim_{\CC} H^k(M^*)=\chi^*$ and $\dim_{\CC} H^k(M_z)=\chi_z$.
    
    It remains to prove that the \eqref{eq:degeneration} has a non-trivial kernel. Let us consider a set \hbox{$\tilde{I}=\{ i_0,i_1,\dots,i_k\}\subset\{0,1,\dots,n\}$} of cardinality $k+1$ with $i_0<\cdots<i_k$.
    We set 
    \begin{equation*}
        \omega_{\tilde{I}}\,:=\,
        \begin{cases}
            e_{i_1}\cdots e_{i_k}&i_0=0\\
            (e_{i_1}-e_{i_0})\cdots (e_{i_k}-e_{i_{0}}) &i_0\neq 0.
        \end{cases}
    \end{equation*}
    Since $E^{\,{\rm red}}_\chi(z)=0$, there exists a pair $(I,J)\in \{ 0,\dots,k\}\times\{k+1,\dots,n\}$ such that $ |I| = |J|$ and $\det(z_{I,J})=0$ at $z$.
    We set $\tilde{I}:=(\{ 0,\dots,k\}\setminus I)\cup J$.
    Then, $\omega_{\tilde{I}}\in H^k(M^*)$ is in the kernel of the morphism \eqref{eq:degeneration}.
    Indeed, if $i_0=0$, then $\{i_1,\dots,i_k\}$ is a dependent subset of $M_z$ and 
    \begin{equation*}
\varphi(\omega_{\tilde{I}})
        \,=\,e_{i_1\cdots i_k}\,=\,e_{i_1}\partial (e_{i_1\cdots i_k})\,\in\, I(\mathscr{A}_z).
    \end{equation*}
    If $i_0\neq 0$, this means that $\cap_{j=0}^kH_{i_j}\neq\varnothing$ and
    \begin{equation*}       \varphi(\omega_{\tilde{I}})\,=\,e_{i_1\cdots i_k}-e_{i_0i_2\cdots i_k}+\cdots
        \,=\,\partial e_{i_0\cdots i_k}\,\in\, I(\mathscr{A}_z).
    \end{equation*}
    Hence, it is enough to show that $\omega_{\tilde{I}}$ is a non-zero element of $H^k(M^*)$.
    Given $z^*\in Z$ generic, we write $\mathcal{O}_{*}$ for the structure ring of the affine variety ${\CC^k\setminus L(\mathscr{A}_{z^*})}$.
    By abuse of notation we write $h_j$ for the dehomogenization of $h_j(\alpha;z)$ introduced in \S\ref{sec:3.1} with $\alpha_0=1$. 
    We consider the ideal $\mathcal{I}\subset\mathcal{O}_*$ generated by $a_1,\dots,a_k$ where $a_i$ is the coefficient of $d\var_i$ of a one-form     \begin{equation}\label{eq:one form}
        \lambda_1d\log (h_1)+\cdots+\lambda_nd\log (h_n)
    \end{equation} 
    on $\CC^k\setminus L(\mathscr{A}_{z^*})$.
    Then, there is an isomorphism $H^k(M^*)\simeq \mathcal{O}_{*}/\mathcal{I}$.
    Indeed, given an element $e_{i_1}\cdots e_{i_k}\in A^k(M^*)$, we assign the coefficient of $d\var_1\wedge \cdots \wedge d\var_k$ in $d\log (h_{i_1})\wedge\cdots\wedge d\log (h_{i_k})$.
    Let $V$ be the zero set of the one form \eqref{eq:one form}.
    Note that the cardinality of $V$ is $\chi^*>0$ by \cite[Theorem 1]{Huh13}.
    Through this correspondence, $\omega_{\tilde{I}}$ corresponds to $\pm \det(z^*_{I,J})\frac{d\var_1\wedge\cdots\wedge d\var_k}{h_{i_1}\cdots h_{i_k}}$ when $i_0=0$ or to $\pm\det(z^*_{I,J})\frac{d\var_1\wedge\cdots\wedge d\var_k}{h_{i_0}\cdots h_{i_k}}$ when $i_0\neq 0$, where $z^*\in Z$ is a generic point.
    Furthermore, there is an isomorphism $\mathcal{O}_*/\mathcal{I}\ni [f(x)]\mapsto \left( f(a)\right)_{a\in V}\in\CC^V$.
    Clearly, $\omega_{\tilde{I}}$ corresponds to a non-zero vector in $\CC^V$.
\end{proof}

Note that the hypothesis that $\chi_z>0$ is fundamental. Central arrangements, for instance, provide examples with $\chi_z=0$ for all $z\in Z$ and non-trivial polynomial $E_{\chi}^{\,{\rm red}}(z)$. A simple example is given for $k=2,n=1$, and $z = [0 \,\,z_{11} \,\,0]^T$.
Let
$$
E^{\,{\rm red}}_\chi(z)=\prod_{\Delta}\Delta(z),
$$
be the irreducible decomposition of $E^{\,{\rm red}}_\chi(z)$.
On a generic point $z\in V_Z(\Delta)$, the number $\chi_z$ is independent of $z$.
We set $m_\Delta:=\chi^*-\chi_z$.
It is convenient to set
$$
E_\chi(z):=\prod_{\Delta}\Delta(z)^{m_\Delta}
$$
to keep track of the drop of Euler characteristics.

\begin{example}\label{ex: artificial_specialized}
    Consider the hyperplane arrangement introduced in Example~\ref{ex:artificial_pt1}. We compute the Euler discriminant for two choices of subspaces $ Z_1, Z_2 \subset \CC^{12}$. In particular, let $ Z_1$ be parametrized by the matrix
    \begin{equation}
       C_1 = \begin{bmatrix}
            w_1+w_2 & 1 & 0 & 0\\
            1 & 0 & 1 & w_2+w_3 \\
            0 & w_1-w_3 & w_1+w_2+w_3 & 1
        \end{bmatrix},
    \end{equation}
    and $ Z_2 =  Z_1 \cap \{w_1+w_2+w_3 = 0\}$. Then, since $ Z_1$ intersects transversally the subspace defined by the matrix $z_G$ in Example~\ref{ex:artificial_pt1}, we have   
    \begin{align*}
    E_\chi( Z_1) \, =\,& (w_1+w_2)^3(w_1-w_3)^3(w_2+w_3)^2(w_1+w_2+w_3)^2\\&(w_1w_2+w_1w_3+w_2^2+2w_2w_3+w_3^2-1)^2\\
    &(w_1^2 + w_1 w_2 - w_1 w_3 + w_1 - w_2 w_3 + w_2 + w_3)\\ &(w_1^2 w_2 + w_1^2 w_3 + w_1 w_2^2 - w_1 w_3^2 - w_2^2 w_3 - w_2 w_3^2 + 1).
    \end{align*}
    Note that the multiplicities of the components remain unchanged from those in \eqref{eq:PAD_example}, consistently with \cite[Theorem 3.3]{esterov2013discriminant}. To compute these exponents using \cite[Theorem 3.3]{esterov2013discriminant}, we employed the function \texttt{EulerDiscriminantQ} from the \texttt{julia} package \texttt{PLD.jl}, \cite{fevola2024principal}. This function returns the value $\chi_{z_Q}$ for each component defined by $\{\Delta_{A\cap Q}(z)=0\}$ of the principal $A$-determinant. For the subspace $ Z_2$, some of the minors in \eqref{eq:PAD_example} identically vanish when plugging in the parametrization. In fact the value of the generic Euler characteristic drops to $\chi^*( Z_2) = 3$, and
    \begin{align*}
    E_\chi( Z_2) \, =\,&  (w_1+w_2)^2(w_1-w_3)^2(w_2+w_3)^2\\
    &(w_1^2 w_2 + w_1^2 w_3 + w_1 w_2^2 - w_1 w_3^2 - w_2^2 w_3 - w_2 w_3^2 + 1).
    \end{align*}
    Note that, in this case, the exponents are not preserved in contrast to those in \eqref{eq:PAD_example}.
\end{example}

\section{Euler discriminant and singular locus}\label{sec:SingularLocus}
In this section, we discuss a relation between Euler discriminants and singular locus of a $\mathcal{D}$-module which an Euler integral underlies.
Note that $f_i$ below do not have to be linear polynomials.
We use a standard notation of $D$-modules as in \cite{hotta2007d}.
In this section, we assume $ Z$ is a smooth algebraic subvariety of $\C^{(k+1)(n-k)}$.

\subsection{D-modules and Euler discriminant}
We first prove some general propositions on $D$-modules and then relate it to Gauss-Manin connection.

Let $D_ Z$ be the sheaf of differential operators on $ Z$.
Given a coherent $D_ Z$-module $M$, its characteristic variety ${\rm Ch}(M)$ is canonically defined as a closed conic subvariety of the cotangent bundle $T^* Z$.
We write $\varpi:T^* Z\to Z$ for the natural projection.
The singular locus ${\rm Sing}(M)$ of $M$ is defined by
\begin{equation*}
    {\rm Sing}(M)\,=\,\varpi\left({\rm Ch}(M)\setminus T^*_Z Z\right),
\end{equation*}
where $T^*ZZ$ is the zero section of $T^*Z$.
For a given point $z\in Z$, we write $\iota_z:\{ z\}\hookrightarrow  Z$ for the canonical embedding.
For a holonomic $D_ Z$-module $M$, we set
\begin{equation*}   \chi_z(M)\,\coloneqq\,\sum_{i=0}^\infty (-1)^i\dim_{\C}{\rm Ext}^i_{D_Z}(M,\mathcal{O}_z),
\end{equation*}
where $\mathcal{O}_z$ is the ring of convergent series at $z$.
We write $\chi^*(M)$ for the holonomic rank of $M$. 
Note that $\chi^*(M)$ is $\chi_z(M)$ for generic $z$.
Thanks to the formula \eqref{eq:5.6}, the notations $\chi_z(M)$ and $\chi^*(M)$ are compatible with $\chi_z$ and $\chi^*$. This observation motivates the following definition 
\begin{equation*}
    \nabla_\chi( Z,M)\,\coloneqq\,\{z\in Z\mid \chi_z(M)<\chi^*(M)\}.
\end{equation*}
Note that $\nabla_\chi( Z,M)$ may not be a closed subvariety of $ Z$ for general $M$.
A simple counter example is given by local cohomology.
Nonetheless, we can prove the following theorem.

\begin{theorem}\label{thm:5.6}
    Let $M$ be a holonomic $\mathcal{D}_Z$-module and let ${\rm Sing}^1(M)$ be the union of the codimension $1$ components of the singular locus of $M$.
    Then, one has the inclusion
    \begin{equation}\label{eqn:codim1 inclusion}
        {\rm Sing}^1(M)\,\subset\, \overline{\nabla_{\chi}( Z,M)}.
    \end{equation}
\end{theorem}

\begin{proof}
Let $z$ be a smooth point of ${\rm Sing}^1(M)$.
The multiplicity of the characteristic cycle at $z$ is at least one.
By Kashiwara's index formula, it follows that $\chi_z(M)<\chi^*(M)$, which implies $z\in \nabla_{\chi}( Z,M)$.
The inclusion \eqref{thm:5.6} follows by taking the Zariski closure.
\end{proof}

\subsection{Gauss-Manin connection}
In this section, we will use the notation and concepts introduced in Section~\ref{sec:2_EulerDiscriminant}.
Let $(s,\nu)=(s_0,\dots,s_\ell,\nu_1,\dots,\nu_n)\in\C^{\ell+1+n}$ be complex parameters.
The total space of variables and coefficients $X:=\{(\var,z)\in(\C^*)^n\times Z\mid \prod_{i=0}^\ell f_i(\var;z)\neq0\}$ is fibered over $ Z$ via the natural projection $\pi:X\to Z$.
For any integer $k$, the space of relative differential $k$-forms $\Omega^k_{X/ Z}(X)$ is given by
\begin{equation*}
    \Omega^k_{X/ Z}(X)\,:=\,\sum_{1\leq i_1<\cdots<i_k\leq n}\mathcal{O}(X)d\var_{i_1}\wedge\cdots\wedge d\var_{i_k}.
\end{equation*}
We set
\begin{equation*}
    \omega\,:=\,d_\var\log\left( f(\var;z)^{-s}\var^\nu\right)=-\sum_{i=0}^\ell s_i\frac{d_\var f_i(\var;z)}{f_i(\var;z)}\,+\,\displaystyle\sum_{i=1}^n\nu_i\frac{d\var_i}{\var_i}\,\in\,\Omega^1_{X/Z}(X).
\end{equation*}
The $1$-form $\omega$ naturally defines the differential
$$\nabla_\omega\,:=\,d_\var\,+\,\omega\wedge\,:\,\Omega^{p}_{X/ Z}(X)\to\Omega^{p+1}_{X/ Z}(X).$$
The cohomology space
\begin{equation*}
    H^n( X/ Z,\omega)\,:=\,\frac{\Omega^n_{X/ Z}(X)}{{\rm Im} \left(\nabla_\omega: \Omega^{n-1}_{X/ Z}(X)\to \Omega^{n}_{X/ Z}(X)\right)}
\end{equation*}
is naturally equipped with a structure of left $D_Z$-module: for a vector field $\partial_z$ on $ Z$ we set 
\begin{equation*}
    \nabla_{\partial_z}^{\rm GM}[\xi(z)]\,:=\,\left[ \frac{\partial\xi}{\partial z}(z)-\sum_{i=0}^\ell s_i\frac{1}{f_i(\var;z)}\frac{\partial f_i}{\partial z}(\var;z)\xi(z)\right],\ \ \ [\xi(z)]\in H^n( X/ Z,\omega).
\end{equation*}
We call the $D_Z$-module $ H^n\left( X/ Z,\omega\right)$ the ($n$-th) Gauss-Manin connection.

Now we consider the case when $M$ is given by the Gauss-Manin connection.
In the language of $D$-modules, it is the so-called ($0$-th) direct image functor:
\begin{equation}\label{eq:direct image}
    H^n(X/ Z,\omega)\,\simeq\, H^0\int_\pi\mathcal{O}_Xf(\var;z)^{-s}\var^\nu\,=:\,M.
\end{equation}
Here, $\mathcal{O}_Xf(\var;z)^{-s}\var^\nu$ denotes the $D_X$-module associated to an integrable connection $\nabla:\mathcal{O}_X\to\Omega^1_X$ given by
$$
\nabla g\,:=\,dg-\sum_{i=0}^\ell s_ig\frac{df_i}{f_i}+\sum_{j=1}^n\nu_jg\frac{d\alpha_j}{\alpha_j}\ \ \ \ (g\in\mathcal{O}_X),
$$
where $d$ denotes the exterior derivative on $X$.
An integral of the form 
\begin{equation}\label{eq:Euler integral}
\int_\Gamma f(\var;z)^{-s}\var^\nu\omega(\var;z)    
\end{equation}
for some cycle $\Gamma$ and $\omega(\var;z)\in\Omega^n_{X/Z}(X)$ is a solution to the $D_{Z}$-module $H^n(X/Z,\omega)$.
The integral in \eqref{eq:Euler integral} is called {\it Euler integral}, see \cite{agostini2022vector} and \cite{matsubara2023four}.

\begin{lemma}\label{lem:vanishing}
For generic $(s,\nu)$, one has the vanishing theorem $    L^j\iota_z^*M=0$ for all $ j\neq 0$.
\end{lemma}

\begin{proof}
By the base change formula and the isomorphism in \eqref{eq:direct image}, it is enough to prove 
\begin{equation}\label{eq:vanishing identity}
    H^{n+j}(X_z,\omega)\,=\,0\ \, \text{ for all }j\neq 0.
\end{equation}
We first construct the sequence of subvarieties $ Z= Z_n\supset  Z_{n-1}\supset\cdots \supset Z_0 $ such that ${\dim_{\C} Z_i\leq i}$, $ Z_i\setminus Z_{i-1}$ is smooth and there exists a smooth variety $X_i$ such that there exists an inclusion $i:\pi^{-1}( Z_i\setminus Z_{i-1})\to X_i$, which gives a commutative diagram
\begin{equation}
\xymatrix{
    \pi^{-1}( Z_i\setminus Z_{i-1})\ar[r]^-{\pi}\ar[d]& Z_i\setminus Z_{i-1}\\
    X_i\ar[ur]&
}
\end{equation}
Here, $X_i$ is equipped with a simple normal crossing divisor $D_i=\sum_{j}D_{ij}$ such that any finite intersection of $D_{ij}$ is transversal to a fiber of $\pi$ and $X_i\setminus D_i=\pi^{-1}( Z_i\setminus Z_{i-1})$.
Such a sequence can be constructed as in \cite[6.15]{deligne2006equations}.
Indeed, one can construct $ Z_{n-1}$ and $X_n$ with the desired property.
Then, we take the smooth locus of $ Z_{n-1}$ and proceed inductively.
Moreover, the construction of \cite[6.15]{deligne2006equations} combined with the argument of \cite[2.3]{Huh13} proves that there is a linear form $\ell_{ij}(s,\nu)\neq 0$ so that $\exp\{ 2\pi\sqrt{-1}\cdot \ell_{ij}(s,\nu)\}$ is the local monodromy around $D_{ij}$ of the local system that $f(\var;z)^{-s}\var^\nu$ defines on a fiber $X_z$.
The identity \eqref{eq:vanishing identity} follows from the same argument as \cite[Appendix]{agostini2022vector}.
The genericity assumption for the parameters is that $\ell_{ij}(s,\nu)\notin\Z$.
\end{proof}
It follows from Lemma \ref{lem:vanishing} that $\nabla_{\chi}( Z)=\nabla_{\chi}( Z,M)$.
Indeed, one has 
$$\mathbb{R}{\rm Hom}_{D_Z}(M,\mathcal{O}_z)\,=\,\mathbb{R}{\rm Hom}_{D_Z}(M,\widehat{\mathcal{O}}_z)\,=\,\mathbb{R}{\rm Hom}_{D_Z}(\mathbb{L}\iota_z^*M,\C),$$
where the first equality is a consequence of regular holonomicity of $M$ and the comparison theorem \cite[Proposition 7.3.1]{hotta2007d}.
Thus, it follows that 
\begin{equation}\label{eq:5.6}
\chi_z(M)\,=\,\sum_{i=0}^\infty (-1)^i\dim_{\C}{\rm Ext}^i_{D_Z}(M,\mathcal{O}_z)\,=\,\dim_{\C}(\iota_z^*M,\C)\,=\,\dim_{\C}H^n(X_z,\omega)\,=\,\chi_z.    
\end{equation}

\begin{theorem}\label{thm:Euler discriminant and singular locus}
    For generic $(s,\nu)$, we have $\nabla_\chi^{1}( Z)={\rm Sing}^1(M).$
\end{theorem}

\begin{proof}
    From Theorem \ref{thm:5.6}, we have $\nabla_\chi^{1}( Z)\supset{\rm Sing}^1(M)$.
    The general theory of $D$-module shows that a $D$-module is an integrable connection outside its singular locus.
    This means that the Euler characteristics does not change on the nonsingular locus.
    This proves the inclusion $\nabla_\chi( Z)\subset{\rm Sing}(M)$.
\end{proof}

\subsection{Arrangement case}

Now let us assume that the homogeneous degrees of polynomials $f_i$ all equal to one.

\begin{theorem}\label{thm:5.1}
For generic $(s,\nu)$, one has an identity
\begin{equation}
    {\rm Sing}(M)=\nabla_\chi( Z).
\end{equation}
    
\end{theorem}

\begin{proof}
    Since $\nabla_\chi( Z)$ is purely one codimensional, it is enough to prove the inclusion ${\rm Sing}(M)\subset\nabla_\chi( Z)$ in view of Theorem \ref{thm:Euler discriminant and singular locus}.
    Let us define an open subvariety $U\subset  Z$ by $U= Z\setminus \nabla_\chi( Z)$.
    We prove the inclusion ${\rm Ch}(M|_U)\subset T^*_UU$ where $M|_U$ is the restriction of $M$ onto $U$.
    We construct a relative compactification of $\pi^{-1}(U)$ with respect to the projection $\pi: \pi^{-1}(U)\to U$.
    We consider the process of iterated blowing-ups (\cite[Theorem 4.2.4]{orlikarrangements},\cite[\S 10.8]{varchenko1995multidimensional}).
    The result is a blowing-up $\widetilde{\pi^{-1}(U)}$ of ${\pi^{-1}(U)}$ equipped with a normal crossing divisor $D$.
    The associated morphism $\widetilde{\pi^{-1}(U)}\to U$ is transversal with respect to the stratification that $D$ defines.
    The result now follows from a standard lemma on the direct image. The following lemma follows from, e.g., \cite[Theorem 4.27]{kashiwara2003d}.
\end{proof}

\begin{lemma}
    Let $Y, Z$ be smooth varieties and $f: Y\to Z$ be a proper morphism.
    Assume that a normally crossing divisor $D=\cup_{j=1}^N D_j$ of $Y$ is given and $f$ is transversal to the stratification that $D$ defines.
    Let $N$ be a $D_Y$-module. For a subset $J$ of $\{ 1,\dots,n\}$, we set $D_J:=\cap_{j\in J}D_j$.
    If $${\rm Char}(N)\subset \cup_{J\subset\{ 1,\dots,N\}}T^*_{D_J}Y,$$
    then $\int_fN$ is an integrable connection. Here, $T^*_{D_J}Y$ denotes the conormal bundle of $D_J$.
\end{lemma}

\appendix
\section{Cosmological integrals in a nutshell}\label{appendix}
Theoretical cosmology aims at understanding the origin, evolution, and large-scale structure of the universe. An essential part of cosmology is knowing how matter and energy are distributed throughout the universe at any given time. A remarkable example of this is the \textit{Cosmic Microwave Background} (CMB). \textit{Cosmological correlators} are the statistical quantities that model such distributions. There are different ways of calculating cosmological correlators; the one we are interested in is the \textit{wavefunction formalism}. It consists of two main steps: computing the wavefunction and then using it to calculate the correlators. In perturbation theory, it is sufficient to compute a finite set of \textit{wavefunction coefficients} that characterize the wavefunction completely up to the perturbative order one is interested in.

This appendix gives a self-contained introduction to the \textit{wavefunction coefficients} in flat space, aiming at explaining how they lead to physically meaningful examples of Euler integrals of products of linear forms \cite{aomoto1997twisted}.

Let $\mathcal{G}_{L,\n}$ denote the set of connected graph with $\n$ external edges and $L$ loops. Then, in perturbation theory, the $\n$-th wavefunction coefficient can be expanded in an infinite sum over all possible Feynman diagrams $G \in \mathcal{G}_{L,\n}$ given by
\begin{equation}\label{eq:waveefunction}
\Psi_\n \, = \, \sum_{\substack{L\geq 0\\G\in \mathcal{G}_{L,\n}}}\psi_G,
\end{equation}
where the $\psi_G$ are integrals that become some expression of the physical parameters after the integration.
In cosmology, when drawing Feynman diagrams, the $\n$ external edges connect the vertices in $V$ to a horizontal line representing the \textit{boundary}. 
See \cite[Section~1]{arkani2017cosmological} for some illustrations of Feynman diagrams for $\n=2,3,4,5$. 
The method to associate an integral to each diagram is given by the \textit{Feynman rules}. For an explicit definition of these rules we refer to \cite[Section~2.2]{arkani2023}. 
For physicists, it is important to note that we work with \textit{conformally coupled scalar fields}. This enables the mathematical approach to capture the essential behavior of the system while remaining computationally feasible.

Under these assumptions, it is convenient to draw a simplified version of the Feynman diagrams where we only keep internal vertices and edges and truncate the $\n$ edges connecting to the boundary. In the following, let $G=(V,E)$ be a connected undirected graph, where $V$ is the set of $n$ vertices of $G$ and $E$ is its finite collection of edges. We write $V(G)$ and $E(G)$ when we want to emphasize that $V$ and $E$ are, respectively, the vertex and edge set of $G.$ 

To the vertices and edges of $G$, we associate complex parameters representing energies, $X = (X_i)_{v_i\in V}$ and $Y = (Y_{ij})_{ij\in E}$. Figure~\ref{fig:Feynman_diag} illustrates some of the graphs we will use in our examples. In physics, the $(X,Y)$-variables are real and positive, but for practical applications it is useful to think of them in larger domains.

\begin{figure}[h]
\includegraphics[]{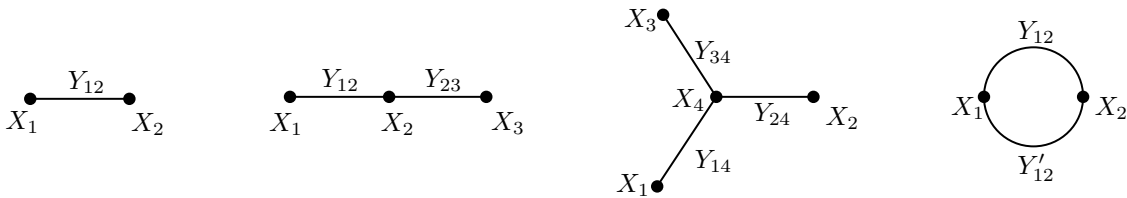}
\caption{Two-site chain, three-site chain, four-site star, and one-loop bubble.}\label{fig:Feynman_diag}
\end{figure}

The first ingredient towards introducing Euler integrals arising in cosmology are the integrands of the terms in the expansion \eqref{eq:waveefunction}. By abusing notation, we will still refer to these as wavefunction coefficients in flat space and denote them $\psi_G$. However, in what follows, the $\psi_G$ will be rational functions of the parameters $(X,Y)$.  The next subsections provide a concise review of two alternative methods to derive the flat space wavefunction coefficients: the \textit{recursion formula} and the \textit{cosmological polytope} \cite{arkani2017cosmological, juhnke2023triangulations, kuhne2022faces}. 

\subparagraph{Recursion Formula}
The recursion formula, introduced in \cite{arkani2017cosmological}, represents the wavefunction coefficient $\psi_G$
 associated with a graph $G$ as a combination of  wavefunction coefficients of subgraphs obtained by sequentially deleting individual edges from $G$. The core idea is to reduce to computing only wavefunction coefficients of graphs that consist of a single vertex and no edges. In this case, in fact, if $X$ is the vertex's energy variable, the wavefunction coefficient is $1/X$, as explained in 
 \cite[Equation~(2.12)]{arkani2017cosmological}. Then, the recursion formula is
\begin{equation}\label{eq:recursion}
\psi_G(X,Y) \, = \, \frac{1}{\sum_{k\in V(G)} X_k}\cdot \sum_{ij\in E(G)} \left(\psi_{H_i}\cdot\,\, \psi_{H_j}\right),
\end{equation}
where $H_i,H_j$ are the subgraphs containing the vertex $v_i$, (resp. $v_j$), obtained from $G$ by deleting the edge $e=ij$. Furthermore, the variable $X_{i}$ (resp. $X_{j}$) attached to $v_i$ in $H_i$  (resp. $v_{j}$ in $H_j$) becomes $X_{i}+Y_{ij}$ (resp. $X_{j}+Y_{ij}$). The derivation of this formula is explicitly given in \cite[Section 2.3]{arkani2017cosmological}.
\begin{example}\label{eq:2FSW}
    Let $G$ be the two-site chain illustrated in Figure~\ref{fig:Feynman_diag}. We have $n=2$, and the energies associated to vertices and edges are $(X_1,X_2,Y_{12})$. We denote $\psi_2$ the associated wavefunction coefficient. The recursion formula is resolved in one step:
    \[
        \psi_2 \, = \, \frac{1}{X_1+X_2} \left(\psi\Bigl(\treeOne\Bigr) \cdot \psi \Bigl(\treeBis\bigr)\right) 
        \,=\, \frac{1}{(X_1+X_2)(X_1+Y_{12})(X_2+Y_{12})} 
    \]
\end{example}

\begin{example}\label{eq:3FSW}
In this example, we show an application of the recursion formula \eqref{eq:recursion} to compute the flat space wave function $\psi_3$ associated with the three-site chain showed in Figure~\ref{fig:Feynman_diag}. Let $1/X\coloneqq\frac{1}{X_1+X_2+X_3}$, then we have
\begin{align*}
\psi_3 \, &= \, \frac{1}{X}\left( \psi\Bigl(\treeOne\Bigr) \cdot \psi \Bigl(\treeTwo\Bigr) + \psi\Bigl(\treeThree\Bigr) \cdot \psi\Bigl(\treeFour\Bigr)\right) \\
&=\,\frac{1}{X}\biggl( \frac{1}{X_1+Y_{12}} \cdot \psi\Bigl(\treeTwo\Bigr) +\psi\Bigl(\treeThree\Bigr) \cdot\frac{1}{X_3+Y_{23}}\biggr)\\
&= \, \frac{1}{X}\biggl(\frac{1}{(X_1+Y_{12})(X_2+X_3+Y_{12})}\cdot \psi\Bigl(\treeFive\Bigr)\psi\Bigl(\treeSix\Bigr) \\&\,\quad+
\frac{1}{X_1+X_2+Y_{23}}\cdot \psi\Bigl(\treeSeven\Bigr)\psi\Bigl(\treeEight\Bigr) \cdot \frac{1}{X_3+Y_{23}}
\biggr)\\
&=\, \, \frac{1}{X}\biggl(\frac{1}{(X_1+Y_{12})(X_2+X_3+Y_{12})(X_2+Y_{12}+Y_{23})(X_3+Y_{23})} \\&\,\quad+
\frac{1}{(X_1+X_2+Y_{23})(X_1+Y_{12})(X_2+Y_{12}+Y_{23})(X_3+Y_{23})}\biggr).\\
\end{align*}
By finding a common denominator, we obtain the following rational function:
\begin{equation*}
\frac{X_1+X_3+2X_2+Y_{12}+Y_{23}}{\textcolor{mycolor1}{(X_1+X_2+X_3)}\textcolor{mycolor2}{(X_1+Y_{12})}\textcolor{mycolor3}{(X_2+Y_{12}+Y_{23})}\textcolor{mycolor4}{(X_3+Y_{23})}\textcolor{mycolor5}{(X_1+X_2+Y_{23})}\textcolor{mycolor6}{(X_2+X_3+Y_{12})}}
\end{equation*}
The use of different colors for the linear forms in the denominator will enhance the clarity of Example~\ref{ex:Euler_3site} in the next section.
\end{example}

\subparagraph{Cosmological Polytope}
An alternative way to introduce the wavefunction coefficients is via the canonical form of a \textit{positive geometry} \cite{arkani2017positive}. For the definition of positive geometry and its canonical form we refer to \cite[Definition~1]{lam2022invitation}.
The cosmological polytope $\mathcal{P}_G$ of a graph $G$ as above plays the role of a positive geometry which is built ad hoc to have as canonical form the differential form 
\[\Omega_{\mathcal{P}_G}\, = \, \psi_G(X,Y) \d X \wedge\d Y,\]
where $\d X = \d X_1 \wedge \dots \d
X_n$ and $\d Y = \wedge_{ij\in E} \d Y_{ij}$. 
For the general definition of this polytope as the convex hull of its vertices we refer to \cite{kuhne2022faces}. The facets of the cosmological polytope were characterized by Arkani-Hamed, Benincasa, and Postnikov in \cite{arkani2017cosmological}. Such a characterization is resumed in \cite[Theorem 2.1]{kuhne2022faces}, where a complete description of the higher codimensional faces is also given.

The cosmological polytope $\mathcal{P}_G$ associated with the graph $G=(V,E)$ lives in the space $\RR^{|E|+|V|-1}$, with standard basis vectors $X=(X_1,\dots,X_n)$ and $Y=(Y_{ij})_{ij\in E}$. The facets of $\mathcal{P}_G$ are in bijection with the connected subgraphs $H = (V(H),E(H))$, where a subgraph is another graph formed from a subset of the vertices and edges of $G$ where all endpoints of the edges of $H$ are in the vertex set of $H$. In particular, the facet $F_H$ is the intersection of $\mathcal{P}_G$ with the hyperplane
\begin{equation}\label{eq:facets_equations}
    \sum_{v\in V(H)} X_v + \sum_{\substack{e\,=\,ij,\\ i\in V(H),\, 
    j\not\in V(H)}} Y_{ij} + \sum_{\substack{e\,=\,ij\not\in E(H),\\ i\in V_H,\, 
    j\in V(H)}} 2Y_{ij} \,=\, 0.
\end{equation}
The facet $F_G$ associated with the full graph $G$ and determined by the hyperplane defined by $\sum_{i=1}^n X_i=0$ is called the \textit{scattering facet.} One can immediately check that the hyperplanes appearing in the denominator of the rational function from Example~\ref{eq:3FSW} are precisely the ones obtained by computing~\eqref{eq:facets_equations} for each subgraph of the three-site chain. It is an open problem to prove that the flat space wavefunction coefficient of a cosmological graph $G$, or more in general a graph, is a rational function with poles on the facet hyperplanes of the cosmological polytope. However, a challenge that arises when computing the canonical form of the cosmological polytope is the computation of the adjoint hypersurface. This is the polynomial that encodes the zeros of the flat space wave function. For more details about the adjoint of a polytope see \cite{kohn2020projective}. This challenge motivates the study of triangulations of the cosmological polytope \cite{juhnke2023triangulations}.

\begin{example}
    The cosmological polytope of the two-site chain is a triangle in $\RR^3$ with coordinates $[X_1,X_2,Y]$ cut out by the hyperplanes $\{X_1+X_2,X_1+Y_{12},X_2+Y_{12}\}$. The adjoint hypersurface in this case is $1$.
\end{example}

The $f$-vector of the cosmological polytope can be computed via the recursive formulas in Theorem 4.5 and Corollary 4.6 in \cite{kuhne2022faces}. In this way, one can compute the number of hyperplanes contributing in the arrangement that we will discuss in the next section.

\subsubsection*{Cosmological Integrals}
Euler integrals of linear arise in the study of cosmological integrals. Closely analogous to Feynman integrals in particle physics, these objects were introduced in \cite{arkani2023} for trees, and later their definition was also extended to graphs with cycles in \cite{hang2024note, baumann2024kinematic}. Starting from the flat space wavefunction coefficients $\psi_G$, introduced in the previous section, we shift the variables $X_i$ associated to the vertices by new complex variables $\alpha_i$ for $i=1,\dots,n$ which will be the integration variables. Furthermore, an additional factor  $\alpha^\varepsilon=\alpha_1^{\eps}\cdots \alpha_n^{\eps}$, where $\eps\in \CC$, together with the product of the edge variables are added to the integrand. The choice of $\varepsilon$ parametrizes the cosmology in which the process takes place. 
This procedure gives the following integral 
\begin{equation}\label{eq:FRW}
I_G(X,Y,\eps) \, = \,	\int_{\Gamma}  2^{n-1}\cdot \prod_{ij\in E} Y_{ij} \cdot \alpha^\varepsilon\cdot \psi_G(X_1+\alpha_1,\dots,X_n+\alpha_n,Y)\, \d \alpha,
\end{equation}
where $\d \alpha = \d \alpha_1\wedge \dots \wedge \d \alpha_n$, and $\Gamma$ is a twisted cycle on the complement of hyperplanes arrangement in $(\CC^*)^n$ defined by the linear forms appearing in the denominators of the function $\psi_G(X+\alpha,Y)$, with $X+\alpha=(X_1+\alpha_1,\dots,X_n+\alpha_n)$.
In the context of cosmology, $\Gamma$ is the positive orthant.
More explicitly, given 
\[\psi_G(X+\alpha,Y) \,\, = \,\, \frac{P(X+\alpha,Y)}{L_1(X+\alpha,Y)\cdots L_r(X+\alpha,Y)},\]
where $P,L_i\in \CC[X,Y]$ are respectively the adjoint hypersurface and the facets of the cosmological polytope, as defined in the previous section, the twisted cycle $\Gamma$ lives in the top-dimensional twisted de Rham homology group defined on the complement of hyperplanes 
\begin{align}\label{eq:hyperplane_arrangement}
\mathcal{H}_G^{(X,Y)}\,&=\,(\CC^*)^{n}\setminus V_{(\CC^*)^{n}}(L_1(X+\alpha,Y)\cdots L_r(X+\alpha,Y))\nonumber\\
&= \, \{\alpha\in (\CC^*)^n \,\, |  \,\, L_1(X+\alpha,Y)\neq0, \,\, i=1,\dots,r\},
\end{align}
for a generic choice of parameters $(X,Y)\in \CC^{|V|+|E|}$. For the more deteailed physics motivation behind the definition of \eqref{eq:FRW} we refer to \cite[Section 2.4]{arkani2023}. 

\begin{example}\label{ex:two-site_integral}
      	For $n=2$, the function $\psi_2(X_1,X_2,Y_{12})$ gives the integral
      	$$I_2(X_1,X_2,Y_{12},\eps) \, = \, \int_{\Gamma} \frac{2 Y_{12} \cdot \alpha_1^\eps \alpha_2^\eps}{\textcolor{mycolor1}{(X_1+X_2+\alpha_1+\alpha_2)}\textcolor{mycolor2}{(X_1+Y_{12}+\alpha_1)}\textcolor{mycolor3}{(X_2+Y_{12}+\alpha_2)}}\d \alpha_1 \wedge \d \alpha_2.$$
The associated hyperplane arrangement is represented in \cite[Figure~1]{arkani2023}.       
\end{example}

By construction, the integrals described in \eqref{eq:FRW} are Euler integrals of product of linear forms. Therefore, we can use the machinery of Sections \ref{sec:PAD_sparse} and \ref{sec:Euler_discriminant} to investigate their singular locus. Given a graph $G$, we denote $z_G$, or $z_n$ if the number of vertices determines $G$, the matrix in \eqref{eq:C(z)} associated to the hyperplanes in \eqref{eq:hyperplane_arrangement}. Then, $z_G$ is a matrix of size $(n+1)\times r$, and we denote $B_G$ the associated bipartite graph introduced in Section \ref{sec:PAD_sparse}. 
\begin{lemma}
    Given a connected undirected graph $G$, then the bipartite graph $B_G$ is also connected. 
\end{lemma}

\begin{proof}
    Consistently with Section \ref{sec:PAD_sparse}, we denote $V_1$ and $V_2$, respectively the left and right vertex set of the bipartite graph $B_G$. The set $V_1$ contains $n+1$ vertices representing the homogeneous coordinates $\alpha_0,\dots,\alpha_n$, while the set $V_2$ has size $r$ with each vertex representing an hyperplane $L_i$. We distinguish the following two cases:
    \begin{itemize}
   \item For any pair of vertices $v_i,v_j\in V_1$, a path between them is given by two edges connecting each $v_i$ to the vertex representing the hyperplane $L_1=X_1+\dots+X_n+\alpha_1+\dots+\alpha_n$ given by shifting the scattering facet;
   \item For any pair of vertices $v_i,v_j\in V_2$, a path between them is given by the edges corresponding to the non-zero coefficient of the variable $\alpha_0$ of the hyperplanes $L_i,L_j$.
   \end{itemize} 
Therefore, given any pair of vertices $v_i,v_j\in V_1\cup V_2$ there exists a path among them.   
\end{proof}

\begin{example}
This example illustrates in details the computation of the principal $A$-determinant of the sparse arrangement induced by the integral of the two-site chain (Figure \ref{fig:Feynman_diag}, left) from Example~\ref{ex:two-site_integral} and its Euler discriminant. In the notation of Section~\ref{sec:PAD_sparse}, we have $k=2$ and $n=5$. The matrix $z_2$, the bipartite graph $B_G$, and the matrix $A$, whose columns give the vertices of the edge polytope $P_G$ are illustrated in Figure~\ref{fig:data_2site}.
\begin{figure}[H]
\begin{subfigure}[T]{0.33\linewidth}
\centering
\vspace{0.7cm}
$$z_2 \,\, = \,\,
    \begin{bmatrix}
        z_{03} & z_{04} & z_{05}\\
        z_{13} & z_{14} & 0\\
        z_{23} & 0 & z_{25}\\
    \end{bmatrix}$$
\end{subfigure}\hfill
\begin{subfigure}[T]{0.33\linewidth}
\centering
\vspace{0.7cm}
\includegraphics[]{figures/bipartite_graph_2.tikz}
\end{subfigure}
\begin{subfigure}[T]{0.33\linewidth}
    $$A \,=\, \begin{bmatrix}
    1 & 1 & 1 & 0 & 0 & 0 & 0\\
    0 & 0 & 0 & 1 & 1 & 0 & 0 \\
    0 & 0 & 0 & 0 & 0 & 1 & 1 \\
    1 & 0 & 0 & 1 & 0 & 1 & 0 \\
    0 & 1 & 0 & 0 & 1 & 0 & 0 \\
    0 & 0 & 1 & 0 & 0 & 0 & 1
\end{bmatrix}$$
\end{subfigure}
\caption{The matrix $z_2$, the bipartite graph $B_G$, and the matrix $A$ for the two-site chain. The colors for the edges of $B_G$ matches the non-zero coefficients in the corresponding hyperplanes $L_1,L_2,L_3$ in the integral from Example~\ref{ex:two-site_integral}.}\label{fig:data_2site}
\end{figure} 
\noindent
The polytope $P_G\subset \RR^6$ is $4$-dimensional, it has f-vector $(7,17,18,8,1)$, and normalized volume $4$. 
Theorem \ref{thm:PAD} provides a closed formula to compute the principal $A$-determinant of the sparse arrangement:
\[E_A(z_2)\,=\,z_{03}z_{04}^2 z_{05}^2 z_{13}^2 z_{14}^2 z_{23}^2 z_{25}^2 (z_{03}z_{14}-z_{04}z_{13})(z_{03}z_{25}-z_{05}z_{23})(z_{03}z_{14}z_{25}-z_{05}z_{14}z_{23}-z_{04}z_{13}z_{25}).\]
The expected degree (see Proposition~\eqref{prop:degree}) is $\deg(E_A(z_2)) = (\dim(P_G)+1)\cdot\text{vol}(P_G) = (4+1) \cdot 4 = 20$. The multiplicities of each factor are computed using the methods from Section~\ref{sec:subdiagram_volume}.
Let us point out that each factor of $E_A(z_2)$ indeed corresponds to a singularity of the Euler integral
\[I_{\Gamma}(z_2,s,\nu)\, = \,\int_{\Gamma} (z_{03}+z_{13}\alpha_1+z_{23}\alpha_2)^{-s_1}(z_{04}+z_{14}\alpha_1)^{-s_2}(z_{05}+z_{25}\alpha_2)^{-s_3} \alpha_1^{\nu_1}\alpha_2^{\nu_2} \frac{\d \alpha_1}{\alpha_1}\wedge\frac{\d \alpha_2}{\alpha_2}.\]
For instance, let us pick the factor $z_{03}$ of $E_A(z_2)$. For simplicity, we assume that the $z_{ij}$ are all real and generic and ${\rm Re}(s_1)<1, {\rm Re}(\nu_1)>-1, {\rm Re}(\nu_2)>-1$.
Let $\Delta(z_2)$ be the bounded chamber specified by three lines $\{\var_1=0\},\{\var_2=0\}$ and $\{ L_1=0\}$.
By a change of coordinates given by $\beta_i=-z_{i3}/z_{03}\alpha_i$ $(i=1,2)$, we obtain that
\begin{align*}
\int_{\Delta(z_2)}L_1^{-s_1}L_2^{-s_2}&L_3^{-s_3}\var_1^{\nu_1}\var_2^{\nu_2}\frac{\d \alpha_1}{\alpha_1}\wedge\frac{\d \alpha_2}{\alpha_2}
=\\
&z_{03}^{\nu_1+\nu_2-s_1}(-z_{13})^{-\nu_1}(-z_{23})^{-\nu_2}\int_{{\Delta}}(1-\beta_1-\beta_2)^{-s_1}f(\beta;z_2)\beta_1^{\nu_1}\beta_2^{\nu_2}\frac{\d \beta_1}{\beta_1}\wedge\frac{\d \beta_2}{\beta_2},
\end{align*}
where ${\Delta}=\{ (\beta_1,\beta_2)\in\mathbb{R}^2\mid \beta_1,\beta_2,1-\beta_1-\beta_2>0\}$ and   $f(\beta;z_2)$ is a holomorphic function in $\beta_1,\beta_2$ with holomorphic parameters $z_2$ coming from a power series expansion of $L_2^{-s_2}L_3^{-s_3}$.
By expanding $f(\beta;z_2)$ as $\sum_{m_1,m_2\geq 0}c_{m_1,m_2}(z_2)\beta_1^{m_1}\beta_2^{m_2}$, the integral is expanded into a function
$$
z_{03}^{\nu_1+\nu_2-s_1}(-z_{13})^{-\nu_1}(-z_{23})^{-\nu_2}
\sum_{m_1,m_2\geq 0}c_{m_1,m_2}(z_2)\frac{\Gamma(1-s_1)\Gamma(1+\nu_1+m_1)\Gamma(1+\nu_2+m_2)}{\Gamma(3-s_1+\nu_1+m_1+\nu_2+m_2)}.
$$
This function is clearly singular for generic values of $s_1,s_2,s_3,\nu_1,\nu_2$ when $z_{03}\to 0$.

To compute the singular locus of the physical integral we need to further restrict to the subspace where the coefficients in $z_2$ are parametrized by the variables $X,Y$ which represent energies. Therefore, the matrix $z_2$ becomes 
$$z_2(X,Y) \, = \,  \begin{bmatrix}
        X_1+X_2 & X_1+Y_{12} & X_2+Y_{12}\\
        1 & 1 & 0\\
        1 & 0 & 1\\
    \end{bmatrix}.$$
Using Theorem~\ref{thm:Euler_discr} we can determine the defining equation of the Euler discriminant:
\[E_\chi(X,Y)\,=\,(X_1+X_2)(X_1+Y_{12})^2(X_2+Y_{12})^2(X_2-Y_{12})(X_1-Y_{12})Y_{12}\]
When describing the singularities of the differential equations annihilating the integral from Example~\ref{ex:two-site_integral}, the singularity $Y_{12}=0$ is discarded, see \cite[Section~3]{arkani2023}. This is due to the normalization factor $Y_{12}$ appearing in the numerator of the integral. 
\end{example}
\begin{example}\label{ex:Euler_3site}
This example illustrates the principal $A$-determinant and the Euler discriminant for the three-site chain. The matrix $z_3$ and the bipartite graph are shown in Figure~\ref{fig3}. 
\begin{figure}[H]
\begin{subfigure}[T]{0.45\linewidth}
\vspace{0.7cm}
\centering
$$z_3 \,\, = \,\,
    \begin{bmatrix}
        z_{04} & z_{05} & z_{06} & z_{07} & z_{08} & z_{09}\\
        z_{14} & z_{15} & 0 & 0 & z_{18} & 0\\
        z_{24} & 0 & z_{26} & 0 & z_{28} & z_{29}\\
        z_{34} & 0 & 0 & z_{37} & 0 & z_{39}\\
    \end{bmatrix}$$
\end{subfigure}
\begin{subfigure}[T]{0.45\linewidth}
\centering
\includegraphics[]{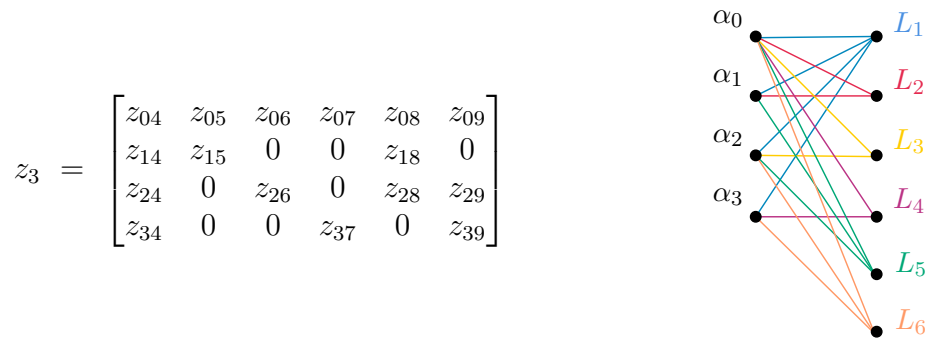}
\end{subfigure}\caption{The coefficients matrix and the bipartite graph for the three-site chain. The colors refer those of the hyperplanes in Example~\ref{eq:3FSW} after introducing the $\alpha$-variables.}\label{fig3}
\end{figure}    
\noindent
Using Theorem \ref{thm:PAD} and the method presented in Section~\ref{sec:subdiagram_volume}, we have that $E_A(z_3)$ equals 
\begin{equation*}
\begin{split}
  &\textcolor{mycolor1}{z_{04}} \textcolor{mycolor1}{z_{05}^{\textcolor{black}{11}}} \textcolor{mycolor1}{z_{06}^{\textcolor{black}{9}}} \textcolor{mycolor1}{z_{07}^{\textcolor{black}{11}}} \textcolor{mycolor1}{z_{08}^{\textcolor{black}{3}}} \textcolor{mycolor1}{z_{09}^{\textcolor{black}{3}}} 
  z_{14}^8 z_{15}^{11} z_{18}^8 z_{24}^4 z_{26}^9 z_{28}^6 z_{29}^6 z_{34}^8 z_{37}^{11}z_{39}^8\textcolor{mycolor1}{(z_{04}z_{15} - z_{05}z_{14})(z_{04} z_{18} - z_{08} z_{14})}\\
  &\textcolor{mycolor1}{(z_{04} z_{26} - z_{06} z_{24})}
  (z_{04} z_{28} -z_{08} z_{24})\textcolor{mycolor1}{(z_{04} z_{29} - z_{09} z_{24})(z_{04} z_{37} - z_{07} z_{34})}(z_{04} z_{39} - z_{09} z_{34}
)\\
&\textcolor{mycolor1}{(z_{08}z_{29} - z_{09}z_{28})}(z_{14} z_{28} - z_{18} z_{24}
)^2(z_{24} z_{39} - z_{29} z_{34}
)^2\textcolor{mycolor1}{(z_{05} z_{18} - z_{08} z_{15})}^3(z_{06} z_{28} - z_{08} z_{26}
)^3\\
&(z_{06} z_{29} - z_{09} z_{26}
)^3\textcolor{mycolor1}{(z_{07} z_{39} - z_{09} z_{37}
)}^3\textcolor{mycolor1}{(z_{05} z_{18} z_{26} + z_{06} z_{15} z_{28} - z_{08} z_{15} z_{26}
)}^3\\
&\textcolor{mycolor1}{(z_{06} z_{29} z_{37} + z_{07} z_{26} z_{39} - z_{09} z_{26} z_{37})}^3\textcolor{mycolor1}{(z_{04}z_{15}z_{26} - z_{05}z_{14}z_{26} - z_{06}z_{15}z_{24})}\\
  &\textcolor{mycolor1}{(z_{04} z_{26} z_{37} - z_{06} z_{24} z_{37} - z_{07} z_{26} z_{34})}(z_{04} z_{15} z_{29} - z_{05} z_{14} z_{29} - z_{09} z_{15} z_{24})\\
  &\textcolor{mycolor1}{(z_{04} z_{15} z_{37} - z_{05} z_{14} z_{37} - z_{07} z_{15} z_{34})}(z_{04} z_{15} z_{39} - z_{05} z_{14} z_{39} - z_{09} z_{15} z_{34})\\
  &(z_{04} z_{18} z_{37} - z_{07} z_{18} z_{34} - z_{08} z_{14} z_{37})(z_{04} z_{18} z_{39} - z_{08} z_{14} z_{39} - z_{09} z_{18} z_{34})\\
  &(z_{04} z_{28} z_{37} - z_{07} z_{28} z_{34} - z_{08} z_{24} z_{37})\textcolor{mycolor1}{(z_{07} z_{28} z_{39} + z_{08} z_{29} z_{37} - z_{09} z_{28} z_{37})}\\
  &\textcolor{mycolor1}{(z_{05} z_{18} z_{29} - z_{08} z_{15} z_{29} + z_{09} z_{15} z_{28})}(z_{14} z_{28} z_{39} - z_{18} z_{24} z_{39} + z_{18} z_{29} z_{34})\\
  &(z_{04} z_{26} z_{39} - z_{06} z_{24} z_{39} + z_{06} z_{29} z_{34} - z_{09} z_{26} z_{34})(z_{04} z_{18} z_{29} - z_{08} z_{14} z_{29} + z_{09} z_{14} z_{28} - z_{09} z_{18} z_{24})\\
  &(z_{04} z_{18} z_{26} + z_{06} z_{14} z_{28} - z_{06} z_{18} z_{24} - z_{08} z_{14} z_{26})(z_{04} z_{15} z_{28} - z_{05} z_{14} z_{28} + z_{05} z_{18} z_{24} - z_{08} z_{15} z_{24})\\
  %%%%%%%%%
  &(z_{04} z_{28} z_{39} - z_{08} z_{24} z_{39} + z_{08} z_{29} z_{34} - z_{09} z_{28} z_{34})(z_{04} z_{29} z_{37} + z_{07} z_{24} z_{39} - z_{07} z_{29} z_{34} - z_{09} z_{24} z_{37})\\
    &\textcolor{mycolor1}{(z_{05} z_{18} z_{29} z_{37} - z_{07} z_{15} z_{28} z_{39} - z_{08} z_{15} z_{29} z_{37} + z_{09} z_{15} z_{28} z_{37})}\\
        &\textcolor{mycolor1}{(z_{04} z_{15} z_{26} z_{37} - z_{05} z_{14} z_{26} z_{37} - z_{06} z_{15} z_{24} z_{37} - z_{07} z_{15} z_{26} z_{34}
)}\\
  &(z_{04} z_{15} z_{26} z_{39} - z_{05} z_{14} z_{26} z_{39} - z_{06} z_{15} z_{24} z_{39} + z_{06} z_{15} z_{29} z_{34} - z_{09} z_{15} z_{26} z_{34})\\
  &(z_{04} z_{15} z_{28} z_{37} - z_{05} z_{14} z_{28} z_{37} + z_{05} z_{18} z_{24} z_{37} - z_{07} z_{15} z_{28} z_{34} - z_{08} z_{15} z_{24} z_{37})\\
  &(z_{04} z_{15} z_{29} z_{37} - z_{05} z_{14} z_{29} z_{37} + z_{07} z_{15} z_{24} z_{39} - z_{07} z_{15} z_{29} z_{34} - z_{09} z_{15} z_{24} z_{37})\\
  &(z_{04} z_{18} z_{26} z_{37} + z_{06} z_{14} z_{28} z_{37} - z_{06} z_{18} z_{24} z_{37} - z_{07} z_{18} z_{26} z_{34} - z_{08} z_{14} z_{26} z_{37})\\
    &(z_{04} z_{18} z_{26} z_{39} + z_{06} z_{14} z_{28} z_{39} - z_{06} z_{18} z_{24} z_{39} + z_{06} z_{18} z_{29} z_{34} - z_{08} z_{14} z_{26} z_{39} - z_{09} z_{18} z_{26} z_{34})\\
   &(z_{04} z_{15} z_{28} z_{39} - z_{05} z_{14} z_{28} z_{39} + z_{05} z_{18} z_{24} z_{39} - z_{05} z_{18} z_{29} z_{34} - z_{08} z_{15} z_{24} z_{39} + z_{08} z_{15} z_{29} z_{34} -\\ &\,\,z_{09} z_{15} z_{28} z_{34}
)
(z_{04} z_{18} z_{29} z_{37} - z_{07} z_{14} z_{28} z_{39} + z_{07} z_{18} z_{24} z_{39} - z_{07} z_{18} z_{29} z_{34} - z_{08} z_{14} z_{29} z_{37} +\\ &\,\,z_{09} z_{14} z_{28} z_{37} - z_{09} z_{18} z_{24} z_{37}).
\end{split}
\end{equation*}
The degree is $\deg(E_A(z_3)) = (\dim(P_G)+1)\cdot\text{vol}(P_G) = (8+1) \cdot 30 = 270$. The factors in cerulean are the ones that do not identically vanish when restricting to the subspace arising in physics.  
The restriction of the coefficients to the physical parameter produces the matrix 
$$  \begin{bmatrix}
        X_1+X_2+X_3 & X_1+Y_{12} & X_2+Y_{12}+Y_{23} & X_3+Y_{23} & X_1+X_2+Y_{23} & X_2+X_3+Y_{12}\\
        1 & 1 & 0 & 0 & 1 & 0\\
        1 & 0 & 1 & 0 &1 &1 \\
        1 & 0 & 0 & 1 & 0 &1
    \end{bmatrix},$$
whose non-vanishing maximal minors give the factors of the Euler discriminant:
\begin{align*}
E_\chi(X,Y)\,=\, &(X_1+X_2+X_3)(X_1+Y_{12})^{10}(X_2+Y_{12}+Y_{23})^9
(X_3+Y_{23})^{10}
(X_1+X_2+Y_{23})^{3}\\
&(X_2+X_3+Y_{12})^3
(X_2+X_3-Y_{12})
(X_3-Y_{23})^6
(X_1+X_3-Y_{12}-Y_{23})
(X_1-Y_{12})^6\\
&(X_1+X_2-Y_{23})
(X_1 - X_3 - Y_{12} + Y_{23})
(X_2-Y_{12}+Y_{23})^3
(X_2+Y_{12}-Y_{23})^3
Y_{12}^6
Y_{23}^6\\
&(X_3-2Y_{12}-Y_{23})
(X_1 - Y_{12} - 2Y_{23})
(X_2-Y_{12}-Y_{23})(X_1 - Y_{12} + 2Y_{23})\\
&(X_3 + 2Y_{12} - Y_{23}) 
(Y_{12}-Y_{23})(Y_{12}+Y_{23}),
\end{align*}
which has degree $77$. The exponents here are computed using the built-in function \texttt{EulerDiscriminantQ} from the \texttt{julia} package \texttt{PLD.jl} as explained in Example~\ref{ex: artificial_specialized}.

\end{example}

\vspace{1cm}
\noindent {\bf Acknowledgements.} 
We are grateful to Simon Telen for insightful discussions and to Henrik J.~Münch for his contribution and computations in the starting stages of the project. We also thank Guilherme L. Pimentel and Tom Westerdijk for their valuable comments on the appendix, Bernd Sturmfels for drawing our attention to  \cite{clarke2023matroid}, and Alexander Postnikov for noting the connection to \cite[Proposition 4.4]{galashin2024trianguloids}.
Lastly, we thank Masahiko Yoshinaga for pointing out the need of the positivity  hypothesis of the signed Euler characteristic in Theorems~\ref{thm:PAD} and \ref{thm:Euler_discr}.

CF has received funding from the
European Union’s Horizon 2020 research and innovation programme under the Marie Sk\l odowska-Curie grant agreement No 101034255); and by the European Research Council (ERC) under the European Union’s Horizon Europe research and innovation programme, grant agreement 101040794 (10000DIGITS).
SJMH is supported by JSPS KAKENHI Grant Number 22K13930, and partially supported by JST CREST Grant Number JP19209317.

{\small
\bibliography{references}}
\bibliographystyle{abbrv}

\vfill 

\noindent{\bf Authors' addresses:}
\smallskip

\noindent Claudia Fevola, Université Paris-Saclay, Inria 
\hfill {\tt claudia.fevola@inria.fr}

\smallskip

\noindent Saiei-Jaeyeong Matsubara-Heo, Kumamoto University\hfill {\tt saiei@educ.kumamoto-u.ac.jp}

\end{document}